\documentclass[journal]{IEEEtran}
\ifCLASSINFOpdf
\else
\fi

\usepackage{mymacros}
\usepackage{graphicx}
\usepackage{multirow}
\usepackage{url}
\usepackage{enumitem}
\usepackage{amsthm}
\usepackage{makecell}

\newtheorem{theorem}{Theorem}[section]
\newtheorem{definition}{Definition}[section]

\newtheorem{lemma}[theorem]{Lemma}
\newtheorem{remarks}{Remark}[section]

\renewcommand{\tilde}{\widetilde}
\renewcommand{\hat}{\widehat}

\usepackage{booktabs}
\usepackage{mathtools}
\mathtoolsset{showonlyrefs}

\usepackage{algorithm}
\usepackage{algpseudocode}

\begin{document}
%
\title{Adaptive Kernel Methods}
%
%
%

\author{Tam\'as~D\'ozsa,~\IEEEmembership{Member,~IEEE},
        Andrea~Angino, Zolt\'an Szab\'o,~\IEEEmembership{Member,~IEEE,}        and~J\'ozsef Bokor,~\IEEEmembership{Fellow,~IEEE},
        Matthias~Voigt
\thanks{T. D\'ozsa and A. Angino contributed equally to this work and should both be considered first authors.}
\thanks{T. D\'ozsa, A. Angino and M. Voigt are with the Faculty 
of Mathematics and Computer Science, UniDistance Suisse, Brig VS, Switzerland \ (e-mail: tamas.dozsa@unidistance.ch, andrea.angino@unidistance.ch, matthias.voigt@fernuni.ch)}
\thanks{T. D\'ozsa, Z. Szab\'o and J. Bokor are with HUN-REN Institute for Computer Science and Control (dozsa.tamas@sztaki.hu, szabo.zoltan@sztaki.hu, bokor.jozsef@sztaki.hu).}
\thanks{Manuscript received XXX.}}

\markboth{IEEE Transactions on Neural Networks and Learning Systems}%
{A. Angino \MakeLowercase{\textit{et al.}}: Adaptive kernel methods}

\maketitle

\begin{abstract}
Kernel methods approximate nonlinear maps in a data-driven manner by projecting the target map onto a finite-dimensional Hilbert space called the solution space. Traditionally, this space is a subspace of a fixed ambient reproducing kernel Hilbert space (RKHS), determined solely by the chosen kernel and the dataset, whose elements identify the basis elements. Consequently, the projection operator underlying the kernel method depends on the loss function, the dataset, and the choice of ambient RKHS. In this study, we consider kernel methods whose solution spaces also depend on learnable parameters that are independent of the dataset. The resulting methods can be viewed as variable projection operators that depend on the loss function, the dataset, and the new learnable parameters instead of a fixed RKHS. This work has two main contributions. First, we propose an efficient approximation of kernels associated with infinite-dimensional RKHSs, commonly used to reduce the solution-space dimension for large datasets. Second, we construct fixed-dimensional, parameter-dependent solution spaces that enable highly efficient kernel models suitable for large-scale problems without the need to approximate kernels of infinite-dimensional RKHSs. Our novel family of adaptive kernel methods generalizes earlier approaches, including Random Fourier Features, and we demonstrate their effectiveness through several numerical experiments.
\end{abstract}


%
\IEEEpeerreviewmaketitle

\section{Introduction}
%
%
%
%
\IEEEPARstart{A}{pproximation} algorithms in reproducing kernel Hilbert spaces (RKHS)~\cite{aronszajn1950theory} using weighted sums of kernel functions are referred to as kernel methods~\cite{kernelMethodsReview} and we refer to approximants characterized by this structure as kernel models. Kernel methods have been successfully used to develop supervised~\cite{chapelle2007training, GaussianProcessReview} and unsupervised machine learning methods~\cite{KernelClustering}. Kernel models satisfy powerful asymptotic convergence guarantees as a consequence of the beneficial properties of RKHSs (see,  e.g.,~\cite{aronszajn1950theory} and section~\ref{sec:probstate}). The usefulness of kernel methods is well-reflected by the number of state-of-the-art methods for control~\cite{hewing2020cautious}, signal and image processing~\cite{kernelMethodsReview}, model reduction and system identification \cite{Wir13,BouH17,BadHMB22}, and other applications~\cite{BERNALDELAZARO2015140, vpsvm} relying on them.

To aid the introduction of our results, we first describe the class of considered learning problems. Denote the space of functions $\C^N \supseteq \cX \to \cY \subseteq \C$ with $N \in \N$ by $\cF$. Let $\cH \subset \cF$ be an RKHS. Kernel methods seek to obtain an approximation to the map-to-be-learned $F \in \cF$ given a dataset $\{ (t_k, F(t_k)) \}_{k=1}^{q}$ with $q \in \N$ in $\cH$. Specifically, kernel methods construct an approximant of $F$ in the form of a kernel model
\begin{equation}
\label{eq:kernelModel}
    f = \sum_{k=1}^q c_k \xi(\cdot, t_k) \in \cH,
\end{equation}
where $c_k \in \C$, $t_k \in \cX \ (k=1,\ldots,q)$ and $\xi : \cX \times \cX \to \C$ is the reproducing kernel associated with $\cH$. Importantly, for a learning problem to be well-defined, we assume that an error (loss) functional $E : \left( \cX \times \cY \times \cY \right)^q \to [0, \infty)$ is given, and the linear parameters $c_k \ (k=1,\ldots,q)$ in Eq.~\eqref{eq:kernelModel} are chosen so that $E((t_1, f(t_1), F(t_1)), \ldots, (t_q, f(t_q), F(t_q)))$ is minimal. The representer theorem~\cite{representer} guarantees the existence of such an optimal kernel model under mild conditions.

Notice that by Eq.~\eqref{eq:kernelModel}, $f$ belongs to an at most $q$-dimensional subspace of $\cH$ spanned by the  functions $\{ \xi(\cdot, t_k) \}_{k=1}^{q}$. In this way, kernel methods can be interpreted as projection operators of the form 
$$
    \cP_{S, E} : \cF \to \cH \quad (S \in \cS, E \in \cE),
$$
where $\cE$ denotes the space of admissible loss functions (as per the representer theorem~\cite{representer}) and $\cS$ denotes the set of all possible datasets $\{ (t_k, F(t_k)) \}_{k=1}^{q} \subset \cX \times \cY$ generated by $F$. Thus, a kernel method can be interpreted as finding a function $f \in \cH$ (or equivalently, coefficients $c_k \in \C \ (k=1,\ldots,q)$) such that
\begin{align}
    \label{eq:proj}
    \begin{split}
    f &:= \cP_{S, E}F \\
      &:= \argmin_{g \in \cH} E((t_1, g(t_1), F(t_1)), \ldots, (t_q, g(t_q), F(t_q))) \\
      &= \sum_{k=1}^q c_k \xi(\cdot, t_k) \in \cH.
    \end{split}
\end{align}

Eq.~\eqref{eq:proj} shows that kernel models are determined by the choice of the loss function $E$, the dataset $S := \{ (t_k, F(t_k) )\}_{k=1}^q$ and the choice of ambient RKHS $\cH$. Indeed, the choice of $\cH$ determines the structure of the basis functions $\xi(\cdot, t_k) \ (k=1,\ldots,q)$ in Eq.~\eqref{eq:kernelModel}. We will refer to the subspace $\cG:= \myspan\{ \xi(\cdot, t_k) : k=1,\ldots,q \} \subset \cH$ as the solution space henceforth. Notice that the at most $q$-dimensional solution space $\cG$ is completely determined by the kernel $\xi(\cdot,\cdot)$ and the dataset. Importantly, the kernel $\xi(\cdot,\cdot)$ and the RKHS $\cH$ are in a one-to-one relationship (see section~\ref{sec:probstate}).

In this study, we generalize the kernel models and kernel methods given in Eqs.~\eqref{eq:kernelModel} and~\eqref{eq:proj}. We refer to our construction as adaptive kernel methods and showcase their effectiveness for several problems. The key idea behind our generalization is to replace the fixed ambient RKHS $\cH$ in Eq.~\eqref{eq:proj} by a set of parameter-dependent RKHSs whose parameters are subject to learning. Formally, consider a parameter domain $\cL$. The proposed adaptive kernel methods and kernel models can then be described with \textit{variable projection operators}~\cite{varpro2} of the form
$$
\cP^{\Lambda}_{S, E}: \cF \to \cH(\Lambda) \quad (\Lambda \in \cL, S \in \cS, E \in \cE),
$$
where the ambient RKHS $\cH(\Lambda)$ is assumed to be well-defined for all $\Lambda \in \cL$ and
\begin{align}
    \label{eq:varproj}
    \begin{split}
    f^{\Lambda} &:= \cP^{\Lambda}_{S, E} F \\
    &:= \argmin_{g \in \cH(\Lambda)} E((t_1, g(t_1), F(t_1)), \ldots, (t_q, g(t_q), F(t_q))) \\
    &= \sum_{k=1}^q c_k \xi^{\Lambda}(\cdot, t_k) \in \cH(\Lambda),
    \end{split}
\end{align}
where $\xi^{\Lambda} \ (\Lambda \in \cL)$ denotes the kernel of the RKHS $\cH(\Lambda)$. Our main objective consists of computing optimal parameters $\Lambda^* \in \cL$ and a corresponding approximant $f^{\Lambda^*}$ function such that
\begin{align}
\label{eq:optimProb}
\begin{split}
  f^{\Lambda^*} &:= \cP^{\Lambda^*}_{S, E} F, \quad \text{where} \\
     \Lambda^* &:= \argmin_{\Lambda \in \cL} \\ 
     & \quad E((t_1, \cP^{\Lambda}_{S, E} F(t_1), F(t_1)), \ldots, (t_q, \cP^{\Lambda}_{S, E} F(t_q), F(t_q))). 
     \end{split}
\end{align}

Therefore, the model can be obtained by a joint optimization over the projection parameter and the weight vector (e.g., using a gradient-based method).

The existence of such a kernel model is discussed in detail in section~\ref{sec:adaptker}. The benefit of the adaptive kernel method given by $\cP^{\Lambda}_{S, E}$, compared to Eq.~\eqref{eq:proj}, is that the solution space $\cG(\Lambda) := \myspan\{ \xi^{\Lambda}(\cdot, t_k)\,:\,k=1,\ldots,q \} \ (\Lambda \in \cL)$ can be adjusted for a more accurate approximation of $F$. The parameterization $\Lambda \mapsto \cH(\Lambda)$ is chosen based on a priori information about $F$ and a correct choice can result in physically interpretable parameters (see, e.g., subsection~\ref{subsec:lti}).

One well-known limitation of traditional kernel methods is that training kernel models becomes computationally expensive, if the map-to-be-approximated ($F \in \cF$) is highly nonlinear~\cite{rahimi, kilic2025interpretable}. In addition, the inference time of such models can also increase significantly~\cite{chapelle2007training, kilic2025interpretable}. These problems manifest if the ambient RKHS $\cH$ is infinite-dimensional~\cite{rahimi}. Because of this, several previous methods attempt to overcome this problem by \textit{approximating the kernel} $\xi(\cdot,\cdot)$ of the RKHS $\cH$ using expansions with a fixed finite set of basis functions (see, e.g.,~\cite{rahimi, CVM, kilic2025interpretable, NEURIPS2018_4e5046fc} and section~\ref{sec:probstate}). One popular approach using this idea is the Random Fourier Features (RFF) method introduced in~\cite{rahimi} by Rahimi and Recht. They leverage Bochner’s theorem, which characterizes the Fourier transform of certain shift-invariant kernels~\cite{rudin2017fourier} to obtain approximations of a large class of kernels. These results have been extended to non-shift-invariant kernels in~\cite{Krein}. Similar constructions are proposed in~\cite{CVM} and~\cite{NEURIPS2018_4e5046fc}. Recently, in~\cite{kilic2025interpretable}, Kilic and Batselier propose a fully stochastic Bayesian tensor network kernel machine model. They utilize tensor product-based kernels to provide solutions for kernel regression problems. Similar results have recently been applied for federated learning tasks~\cite{Fed}.
The adaptive kernel methods proposed in this paper can also be used with large datasets. In particular, if we assume that the reproducing kernels $\xi^{\Lambda}(\cdot,\cdot)$ in Eq.~\eqref{eq:varproj} induce $D$-dimensional RKHSs $\cH(\Lambda) \ (\Lambda \in \cL)$ with $D \ll q$, then the adaptive methods can be trained in a reasonable time, similarly to~\cite{rahimi, CVM, kilic2025interpretable, NEURIPS2018_4e5046fc}. We propose two different methods to obtain such adaptive kernel models. First, following the ideas in~\cite{rahimi}, we attempt to find an appropriate kernel $\xi^{\Lambda}(\cdot,\cdot)$ of a finite-dimensional RKHS $\cH(\Lambda)$ by approximating the kernel of an infinite-dimensional RKHS. We provide a general error formula for this approximation that can be used for any parametrization $\Lambda \mapsto \cH(\Lambda)$. In addition, we give a concrete example of the application of our formula using the Hardy-Hilbert space $H_2(\mathbb{D})$, which is an infinite-dimensional RKHS. Second, we fix the parametrization $\Lambda \mapsto \cH(\Lambda)$ such that for any $\Lambda \in \cL$, we obtain a $D$-dimensional RKHS with $D \ll q$. Then, we solve Eq.~\eqref{eq:optimProb} directly using nonlinear optimization schemes. We verify that the celebrated RFF method~\cite{rahimi} arises as a special case of the proposed algorithm. We illustrate the benefit of using adaptive kernel methods through two numerical experiments. In the first case, we consider the identification of a linear dynamical system and show that choosing the parametrization $\Lambda \mapsto \cH(\Lambda)$ correctly allows for the perfect reconstruction of $F$ with a low-complexity model as well as interpretable learned parameters $\Lambda$. In the second example, we consider a large benchmark machine learning dataset and show that the flexibility provided by adaptive kernel methods has several benefits. Indeed, in our experiments, the proposed methods significantly improve model performance compared to RFF, while maintaining comparable training costs. If the parametrization is chosen correctly, the proposed kernel models learn $F$ in much fewer iterations and at reduced training time.

We summarize the most important novelties of our study:
\begin{enumerate}
    \item We propose adaptive kernel models of the form discussed in Eq.~\eqref{eq:varproj} to solve approximation problems (see sections~\ref{subsec:kernLinAdapt} and~\ref{sec:adaptker}).
    \item We discuss how to obtain adaptive kernels $\xi^{\Lambda}(\cdot,\cdot)$ inducing finite-dimensional RKHSs that approximate the kernel of an infinite-dimensional RKHS well. We provide a generalized error formula and an example in $H_2(\D)$ (see section~\ref{subsec:kernLinAdapt}). 
    \item We propose solving Eq.~\eqref{eq:optimProb} with nonlinear optimization methods.
    We prove the existence of these solutions under weak assumptions (see sections under~\ref{subsec:orthogonal} and~\ref{subsec:GVP}).    
    \item We show that previous kernel approximation techniques, specifically the RFF method~\cite{rahimi}, are a special case of the proposed adaptive kernel models (see section~\ref{subsec:prevMeth}).  
    \item We provide a diverse array of numerical experiments to demonstrate the effectiveness of the proposed adaptive kernel methods (see section~\ref{sec:exp}). All the disclosed results are fully reproducible, see the code and data availability statement.
\end{enumerate}

The rest of this paper is organized as follows. In section~\ref{sec:probstate}, we review some useful properties of reproducing kernel Hilbert spaces and kernel-based approximations. Section~\ref{sec:linker} introduces the class of parameterized kernels inducing finite-dimensional RKHSs and focuses on how these can be used to approximate kernels of infinite-dimensional RKHS. In section~\ref{sec:adaptker}, we discuss how to construct adaptive kernel models in finite-dimensional Hilbert spaces and propose methods to solve Eq.~\eqref{eq:optimProb}. Section~\ref{sec:exp} contains our examples and numerical experiments. Finally, in section~\ref{sec:conc} we draw our conclusions and discuss future research directions.

\section{Properties of kernel methods}
\label{sec:probstate}
In this section, we review some important results from RKHS theory. These results motivate the use of kernel methods for approximating nonlinear maps in a data-driven manner. Then, we address some difficulties associated with using kernel methods to learn large datasets.

We first recall the definition of an RKHS.
\begin{definition}[Reproducing kernel Hilbert space]
\label{def:rkhs}
Let $\cX \subseteq \C^N$ be a non-empty set. Consider the Hilbert space of functions $\cX \to \C$ endowed with the inner product ${\langle \cdot, \cdot \rangle}_{\cH}$. If the linear evaluation functional
$$
   \cF_t:\cH \to \C \quad \text{with} \quad \cF_tf := f(t)
$$
is bounded (and hence continuous) for all $t \in \cX$, then $\cH$ is called a \emph{reproducing kernel Hilbert space}.
\end{definition}
Riesz's representation theorem (see, e.g.,~\cite{hartig1983riesz}) states that for any bounded linear operator $L$ on a Hilbert space, there exists a unique $h_L \in \cH$ such that $Lf = {\langle f, h_L \rangle}_{\cH}$. Using this and Definition~\ref{def:rkhs}, it follows immediately that in the RKHS $
\cH$ we have
\begin{equation}
    \label{eq:kernelSlice}
    \cF_t f = {\langle f, \xi_t \rangle}_{\cH}
\end{equation}
for some unique $\xi_t \in \cH$. This leads to the definition of reproducing kernels (see, e.g.,~\cite{aronszajn1950theory}).
\begin{definition}(Reproducing kernel).
\label{def:repKer}
    Let $\cH$ be an RKHS of functions $\cX \to \cY$. The map
    $$
    \xi : \cX \times \cX \to \C \quad \text{with} \quad \xi(x, t) := \xi_t(x),
    $$
    where $\xi_t$ is defined according to Eq.~\eqref{eq:kernelSlice},
    is called the \emph{reproducing kernel} associated with $\cH$.
\end{definition}
Definition~\ref{def:repKer} implies the so-called reproducing property of the kernel
$
    \xi(x, t) = {\langle \xi_x, \xi_t \rangle}_{\cH}
$
for all $t,x \in \cX$ which also explains the naming convention of these spaces. We note that we shall use the term kernel slices to term the functions $\xi_t(\cdot)$ henceforth. By Riesz's representation theorem, each RKHS has a unique reproducing kernel associated with it. Conversely, Aronszajn's theorem (see, e.g.,~\cite{aronszajn1950theory}) guarantees that any Hermitian positive definite bilinear form uniquely identifies an RKHS. 
Furthermore, this kernel has the following properties:
\begin{theorem}[Properties of reproducing kernels]
    \label{thm:kerProps}
    Let $\cH$ be an RKHS and $\xi(\cdot,\cdot)$ be its reproducing kernel defined according to Definition~\ref{def:repKer}. Then $\xi(\cdot,\cdot)$ has the following properties:
    \begin{enumerate}[label = \alph*)]
        \item (Hermitian symmetry) For any $t,x \in \cX$, 
        $\xi(x, t) = \overline{\xi(t, x)}$.
        \item (Positive definiteness) For any $\{ x_k \}_{k=1}^{q} \subset \cX $ and $\{ c_k \}_{k=1}^{q} \subset \C \ (q \in \N)$, it holds that
        $$
            \sum_{k,{k'}=1}^{q} c_k \overline{c_{k'}} \xi(x_k, x_{k'}) \ge 0,
        $$
        with equality, if and only if $c_1=\ldots=c_q = 0$.
        \item (Reproducing property) For any $f \in \cH$ and $t \in \cX$, 
        $$
        f(t) = {\langle f, \xi_t \rangle}_{\cH} = {\langle f, \xi(\cdot, t) \rangle}_{\cH}.
        $$
    \end{enumerate}
\end{theorem}
For a proof of Theorem~\ref{thm:kerProps} as well as a general and deep discussion on RKHS theory, we recommend~\cite{aronszajn1950theory} and~\cite{paulsen2016introduction}. 
The next few theorems provide motivation for using kernel models to approximate functions. In the following, $\clos(\cdot)$ denotes the closure of a set (w.r.t. the metric induced by the inner product ${\langle \cdot,\cdot \rangle}_{\cH}$)
\begin{theorem}[Density of kernel slices]
    \label{thm:densKer}
    Let $\cH$ be an RKHS with the  kernel $\xi(\cdot,\cdot)$. Then $$G:=\clos\left({\myspan} \left\{ \xi(\cdot, t) \, : \, t \in \cX \right\} \right)  = \cH.$$
\end{theorem}
\begin{proof}
    Suppose $g \in G^{\perp}$. Then, for each $t \in \cX$, $0 = {\langle g, \xi(\cdot, t) \rangle}_{\cH} = g(t)$.
    Hence $G^{\perp} = \{0\}$ and $G = \cH$.
\end{proof}
If the RKHS in Theorem~\ref{thm:densKer} is separable, then a countable dense family of kernel slices exists in $\cH$. Finally we note, that Aronszajn's general existence theorem~\cite{aronszajn1950theory, paulsen2016introduction} states that for any set $\cX$ and function $f : \cX \to \C$, there exists an RKHS $\cH$ which contains $f$.

From a practical point of view, it is important to also mention the representer theorem~\cite{representer}. 
\begin{theorem}[Representer theorem]
\label{thm:representer}
    Let $\cH$ be an RKHS of functions $\cX \to \cY$ with kernel $\xi(\cdot,\cdot)$ and let $\{ (t_k, y_k) \}_{k=1}^{q} \subset \cX \times \cY$. Let $E : \left(\cX \times \cY \times \cY\right)^q \to [0,\infty)$ be an arbitrary error functional. Finally, let $g : [0, \infty) \to \R$ be a strictly increasing function. Then, there exists an  $f^{\ast} \in \cH$ such that
    \begin{multline*}
        f^{\ast} = \argmin_{f \in \cH} \\ \left( E((t_1, f(t_1), y_1), \ldots, (t_q, f(t_q), y_q)) + g({\| f \|}_{\cH}) \right).
    \end{multline*}
    Furthermore, $f^{\ast}$ attains the structure
    $$
        f^{\ast} = \sum_{k=1}^q c_k \xi(\cdot, t_k) \quad (c_k \in \C, \ k=1,\ldots,q).
    $$
\end{theorem}
In other words, given an approximation (learning) problem over a finite dataset, an optimal approximant can always be found in the form of a kernel model. 

\begin{theorem}[Pointwise convergence theorem]
\label{thm:pointwise}
    Let $\cH$ be an RKHS of $\cX \to \cY$ functions. Let $f \in \cH$ and ${(f_n)}_{n=1}^\infty$ be a function sequence in $\cH$ satisfying
    $
        \lim_{n \to \infty} {\| f - f_n \|}_{\cH} = 0.
    $
    Then, 
    $$
        \lim_{n \to \infty} |f(t) - f_n(t)| = 0 
    $$
    for all $t \in \cX$. 
\end{theorem}
Thus, in an RKHS, convergence in norm implies pointwise convergence. The proof of Theorem~\ref{thm:pointwise} can be found in~\cite{paulsen2016introduction}. 

Next, we review Mercer's theorem, which describes the structure of reproducing kernels given by Definition~\ref{def:repKer}. For a deeper discussion and proof, we recommend~\cite{steinwart2012mercer} and~\cite{paulsen2016introduction}. Here we present a simplified version that is tied to the setting of this paper.
\begin{theorem}[Mercer's theorem]
   Assume that the kernel $\xi : \cX \times \cX \to \C$ satisfies Theorem~\ref{thm:kerProps} along with
    $$
        \int_{\cX} \xi(x, x) \,\du x < \infty.
    $$
    Let $L_2(\cX)$ denote the Hilbert space of measurable and square-integrable functions $\cX \to \C$.
    Then, the integral operator
    \begin{equation*}
     T: L_2(\cX) \to  L_2(\cX) \quad \text{with} \quad (Tf) (x) := \int_{\cX} \xi(x, t) f(t)\, \du t
    \end{equation*}
    is a Hilbert-Schmidt operator and consequently compact, self-adjoint, and positive definite on $L_2(\cX)$. It follows that there exist countably many eigenvalues $\lambda_j > 0$ and corresponding  orthonormal eigenfunctions $\phi_j \in L_2(\cX) \ (j \in \N \cup \{0\})$ such that
    \begin{equation}
        \label{eq:MercerExp}
        \xi(x, t) = \sum_{j=0}^{\infty} \lambda_j \phi_j(x) \overline{\phi_j(t)}
    \end{equation}
    for a.e. $(x, t) \in \cX \times \cX$. Furthermore, the induced RKHS is given by
    \begin{equation}
        \label{eq:MercerH}
        \cH := \left\{ f = \sum_{j=0}^{\infty} w_j \phi_j \,: \, \sum_{j=0}^{\infty} \frac{|w_j|^2}{\lambda_j} < \infty  \right\}
    \end{equation}
    with $\cH$-inner product
    \begin{equation*}
    \left\langle \sum_{j=0}^{\infty} w_j \phi_j ,\sum_{j=0}^{\infty} d_j \phi_j\right\rangle_{\cH} := \sum_{j=0}^{\infty} \frac{w_j \overline{d_j}}{\lambda_j}.
    \end{equation*}
\end{theorem}
From Eq.~\eqref{eq:MercerExp} and Eq.~\eqref{eq:MercerH} it follows that furthermore, if $\{\varphi_j\}_{j=0}^{\infty} \subset \cH$ is a complete and orthonormal system in $\cH$, then
\begin{equation}
    \label{eq:ONRexp}
    \xi(x, t) := \sum_{j=0}^{\infty} \varphi_j(x) \overline{\varphi_j(t)}, 
\end{equation}
where convergence holds in the $\cH$-norm (see Eq.~\eqref{eq:MercerH}) and thus pointwise for all $x, t \in \cX$ (see Theorem~\ref{thm:pointwise}). We emphasize that the functions $\{\phi_j\}_{j=0}^{\infty}$ in Eq.~\eqref{eq:MercerExp} are orthonormal in $L_2(\cX)$, while the basis $\{ \varphi_j \}_{j=0}^{\infty}$ in Eq.~\eqref{eq:ONRexp} denote an arbitrary orthonormal system in $\cH$.

In order to properly discuss some problems related to the application of kernel methods, it is beneficial to review the so-called \textit{feature space} representation of kernel models~\cite{jorgensen2025new}. In particular, consider the following definition.
\begin{definition}[Feature map, feature space]
    \label{eq:featSpace}
    Given an RKHS $\cH$ of functions $\cX \to \cY$ with positive definite kernel $\xi(\cdot,\cdot)$, the Hilbert space $\fG$ is called a \emph{feature space}, if there exists a map $\Phi : \cX \to \fG$ such that 
    $$
        \xi(x, t) = {\langle \Phi(x), \Phi(t) \rangle}_{\fG} \quad (x, t \in \cX).
    $$
    The map $\Phi$ is called the \emph{feature map}. Denote the set of all possible pairs of feature maps and feature spaces for the kernel $\xi(\cdot,\cdot)$ by
    \begin{equation}
        \label{eq:feature}
        \fG_{\xi} := \left\{ (\Phi, \fG) \,: \,\xi(x, t) = {\langle \Phi(x), \Phi(t) \rangle}_{\fG} \text{ for all } x,t \in \cX  \right\}.
    \end{equation}
\end{definition}
The set defined in Eq.~\eqref{eq:feature} is never empty~\cite{jorgensen2025new}. For example, the pair $\Phi(x) = \xi_x$ and $\fG = \cH$ is always in $\fG_{\xi}$, because $\xi(x, t) = \langle \Phi(x), \Phi(t) \rangle_{\cH} = \langle \xi_x, \xi_t \rangle_{\cH}$ $(x, t \in \cX)$. Thus, every kernel $\xi(\cdot,\cdot)$ can be written in terms of a Hilbert space inner product given an appropriate pair $(\Phi, \fG) \in \fG_{\xi}$. Eq.~\eqref{eq:ONRexp} gives a non-trivial example, where $\fG := \ell_2$ (the Hilbert space of (unilateral) square-summable complex sequences) and $\Phi(x) := ( \varphi_j(x) )_{j=0}^{\infty} \in \ell_2$. Since the space $\ell_2$ is endowed with the inner product ${\langle f, g \rangle}_{\ell_2} := \sum_{j = 0}^\infty f_j \overline{g_j} \ (f, g \in \ell_2)$, for the $\ell_2$ sequences $\Phi(x) =( \varphi_j(x))_{j=0}^{\infty}$ and $ \Phi(t) = ( \varphi_j(t) )_{j=0}^{\infty}$, we obtain 
$$
    \xi(x, t) = {\langle \Phi(x), \Phi(t) \rangle}_{\ell_2} = \sum_{j=0}^{\infty} \varphi_j(x) \overline{\varphi_j(t)},
$$
which is exactly the formulation given in Eq.~\eqref{eq:ONRexp}.

We now discuss a well-known limitation of kernel methods. Despite all of the practical benefits of kernel methods discussed in this section, approximating a nonlinear function $F : \cX \to \cY$ using a kernel model is often computationally infeasible~\cite{rahimi}. In particular, to ensure that the kernel model approximates $F \in \cF$ well (with respect to a given loss $E$), a large dataset is required. Indeed, the number $q \in \N$ of learnable parameters $c_k \in \C \ (k=1,\ldots,q)$ could explode in the kernel model specified in Eq.~\eqref{eq:kernelModel}. One possibility to overcome this issue is to represent the kernel using an appropriate pair $(\Phi, \fG) \in \fG_{\xi}$, i.e.,
\begin{multline}
    \label{eq:linKer}
    f = \sum_{k=1}^{q} c_k \xi(\cdot, t_k) = \sum_{k=1}^{q} c_k {\langle \Phi(\cdot), \Phi(t_k) \rangle}_{\fG} \\= \left\langle \Phi(\cdot), \sum_{k=1}^{q} \overline{c_k} \Phi(t_k) \right\rangle_{\fG} = {\langle \Phi(\cdot), \overline{w} \rangle}_{\fG}.
\end{multline}
We recall that the number $q \in \N$ of parameters $c_k \in \C \ (k=1,\ldots,q)$ that define the kernel model $f$ is the same as the number of data samples in the dataset. Thus, this number can be very large. On the other hand, if the feature space $\fG$ in Eq.~\eqref{eq:linKer} is of dimension $D \in \N$, where $D \ll q$, then $w \in \fG$ describing the kernel model might more easily be found. Since any $D$-dimensional Hilbert space is isometrically isomorphic to $\C^D$, without loss of generality, we may consider a kernel model $f$ of the form
\begin{multline}
    \label{eq:featureModel}
    f = \sum_{k=1}^q c_k {\langle \Phi(\cdot), \Phi(t_k) \rangle}_{\C^D} = \left\langle \Phi(\cdot), \sum_{k=1}^q \overline{c_k} \Phi(t_k) \right\rangle_{\C^D} \\ = 
    \langle \Phi(\cdot), \overline{w} \rangle_{\C^D} = \sum_{j=0}^{D-1} w_j \varphi_j(\cdot).
\end{multline}
Many prior works~\cite{rahimi, NEURIPS2018_4e5046fc} use feature maps to finite-dimensional Hilbert spaces to mitigate the high dimensionality of the problem. That is, they propose to use kernel models that can be described according to Eq.~\eqref{eq:featureModel} where $D \ll q$. We make use of the notation
\begin{equation}
    \label{eq:truncKer}
    \xi_D(x, t) :=  {\langle \Phi(x), \Phi(t) \rangle}_{\C^D} = \sum_{j=0}^{D-1} \varphi_j(x) \overline{\varphi_j}(t) \quad (x, t \in \cX),
\end{equation}
where $\{ \varphi_j \}_{j=0}^{D-1}$ form a linearly independent function system in a finite-dimensional RKHS. Previous methods~\cite{rahimi} promote the use of feature maps, so that the truncated kernel satisfies $\xi_D(\cdot,\cdot) \approx \xi(\cdot,\cdot)$ in an appropriate metric.

We note that the adaptive kernel models and kernel methods proposed in this study (see Eq.~\eqref{eq:varproj}) improve this idea. In particular, we can write the proposed adaptive kernels from Eq.~\eqref{eq:varproj} in the feature map notation
\begin{equation}
    \label{eq:adaptFeatures}
    \xi^{\Lambda}_D(x, t) = {\langle \Phi^{\Lambda}(x), \Phi^{\Lambda}(t) \rangle}_{\C^D} \quad (x, t \in \cX, \Lambda \in \cL),
\end{equation}
that is, we can define parameter-dependent feature maps to obtain kernels for families of $D$-dimensional RKHSs. Then, obtaining a good approximation of $F \in \cF$ (with respect to the loss $E$) in the  RKHS $\cH(\Lambda)$ induced by $\xi_D^{\Lambda}(\cdot,\cdot)$, one can exploit the flexibility provided by the parameterization $\Lambda \mapsto \cH(\Lambda)$ in several ways. We can, for example, follow the ideas in~\cite{rahimi} and choose $\Lambda \in \cL$, such that $\xi_D^{\Lambda}(\cdot,\cdot) \approx \xi(\cdot,\cdot)$, where $\xi$ is the kernel of an infinite-dimensional RKHS which contains a good approximation of $F \in \cF$. This is discussed in the next section. Alternatively, we can build adaptive kernel models (see Eq.~\eqref{eq:varproj}) using the feature-induced adaptive kernels in Eq.~\eqref{eq:adaptFeatures} and solve Eq.~\eqref{eq:optimProb} directly to obtain scalable kernel methods, which can approximate highly nonlinear target functions $F \in \cF$. We exploit the fact that $D$ in Eq.~\eqref{eq:adaptFeatures} is chosen to ensure that training is computationally feasible in the presence of a large dataset (see Eq.~\eqref{eq:linKer}). Finally, we mention the following theorem from~\cite{invariantMetrik}.
\begin{theorem}[Optimal kernel in nested RKHSs]
\label{thm:nested}
    Let $\cH_1$ and $\cH_2$ be two RKHSs of functions $\cX \to \cY$ induced by the kernels $\xi_1(\cdot,\cdot)$ and $\xi_2(\cdot,\cdot)$, respectively, such that
    $$
        \cH_1 \subset \cH_2.
    $$
    Assume that ${\langle f, g \rangle}_{\cH_1} = {\langle f, g \rangle}_{\cH_2}$ if $f,g \in \cH_1$. Let $X\subseteq \cX$ and $\cH_{m, X} := \myspan\{ \xi_m(\cdot, x) : x \in X\} \subseteq \cH_m$ $(m=1,2)$. For $m=1,2$, consider further the orthogonal projection operators 
    \begin{equation*}
    \cP_{m, X} : \cF \to \cH_{m, X}, \quad \cP_{m, X}f := \argmin_{g \in \cH_{m,X}} \left\| f-g \right\|_{\cH_m}.
    \end{equation*}
    
    Then, for any $F \in \cH_1$ and $X \subseteq \cX$, 
    $$
        {\|F - \cP_{1, X} F\|}_{\cH_2} \leq {\|F - \cP_{2, X} F\|}_{\cH_2}
    $$
    holds.
\end{theorem}
Theorem~\ref{thm:nested} states that if the map-to-be-learned is fully contained in $\cH_1$, then a kernel method projecting to the subspaces of the smaller space $\cH_1$ is always preferable. In other words, the best kernel model is obtained for $F$, if the dimension of the RKHS coinciding with the range of the operator given in Eq.~\eqref{eq:varproj} is as small as possible. This philosophy fits well with the methods proposed in the current study, because for certain choices of the parametrization $\Lambda \mapsto \cH(\Lambda)$ in Eq.~\eqref{eq:varproj} we can obtain a good approximation of $F$, without increasing the dimensionality of $\cH(\Lambda)$ (see section~\ref{sec:adaptker}).

\section{Approximating kernels of RKHSs}
\label{sec:linker}
In this section, we are interested in deriving an error estimate for the approximation
\begin{equation}
\label{eq:apprKer}
    \xi(x, t) \approx \sum_{j=0}^{D-1} \varphi_j(x) \overline{\varphi_j(t)}= \xi_D(x, t) \quad (x, t \in \cX), 
\end{equation}
where $\xi(\cdot,\cdot)$ is the reproducing kernel of the RKHS $\cH$ and $\{\varphi_j\}_{j=0}^{D-1} \subset \cH$ is an orthonormal system.

\subsection{Kernel approximation with different basis expansions}
\label{subsec:kernLinAdapt}
Recall that $\cH$ contains functions $\cX \to \cY$ where $\cX \subset \C^N$ and $\cY \subset \C$. We restrict our investigation to so-called tensor product systems. That is, we assume that for a fixed index $\boldsymbol{j} \in \N^N$, the basis function in Eq.~\eqref{eq:apprKer} corresponding to $\boldsymbol{j}$ is given by
$$
    \varphi_{\boldsymbol{j}}(x_1,\ldots,x_N) = \prod_{v=1}^N \phi_{\boldsymbol{j}_v}(x_v) \quad \left(x = [x_1,\ldots,x_N]^\top \in \cX\right).
$$
Here, we assume that there exists an RKHS $\fH$ of $\C \to \C$ functions generated by the reproducing kernel 
\begin{equation}
    \label{eq:tensprod}
    \zeta(x,t) := \sum_{j=0}^{\infty} \phi_j(x) \overline{\phi_j(t)} \quad (x,t \in \C),
\end{equation}
where $\{ \phi_j \}_{j=0}^{\infty}$ is a complete and orthonormal system in $\fH$. In order to simplify notation, we shall assume lexicographic ordering of the multi indices $\boldsymbol{j} \in \N^N$. That is, for $\boldsymbol{j}, \boldsymbol{n} \in \N^{N}$, we say $\boldsymbol{j} > \boldsymbol{n}$
if ${\| \boldsymbol{j} \|}_1 > {\|\boldsymbol{n}\|}_1$. In case we have ${\| \boldsymbol{j} \|}_1 = {\|\boldsymbol{n}\|}_1$, we still say $\boldsymbol{j} > \boldsymbol{n}$, if for the smallest index $s=1,\ldots,N$ for which $\boldsymbol{j}_s \neq \boldsymbol{n}_s$ we have $\boldsymbol{j}_s > \boldsymbol{n}_s$. This defines a total ordering of the indices $\boldsymbol{j} \in \N^N$. We note that this ordering is not unique and others can also be considered. To simplify discussion, we identify $\{ \varphi_{\boldsymbol{j}} \}_{\boldsymbol{j} \in \N^N} = \{ \varphi_j \}_{j=0}^{\infty}$ using this ordering and assume that the basis functions satisfy the tensor product property in Eq.~\eqref{eq:tensprod}.

In order to obtain a practical error estimate for the approximation~\eqref{eq:apprKer}, we consider the pointwise error
$$
   \left| \xi(x, t) - \xi_D(x, t) \right| \quad (x, t \in \cX),
$$
where the approximation $\xi_D$ is defined according to Eq.~\eqref{eq:truncKer}. We have the following theorem.
\begin{theorem}[Error of kernel approximation]
    \label{thm:kernelErr}
    Let $\cH$ be a separable RKHS of  functions $\cX \to \cY$ $(\cX \subset \C^N, \cY \subset \C, \ N \in \N)$ generated by the reproducing kernel $\xi : \cX \times \cX \to \C$. Let $\{ \varphi_j \}_{j=0}^{\infty} \subset \cH$ be a fixed complete orthonormal system. Then, for any $x,t \in \cX$, we have
    \begin{multline}
        \label{eq:errest}
        \left| \xi(x, t) - \xi_D(x, t) \right| \\ \leq \left(\sum_{j=D}^{\infty} \left| \varphi_j(x) \right|^2\right)^{1/2} \left(\sum_{j=D}^{\infty} \left| \varphi_j(t) \right|^2\right)^{1/2} \quad (D \in \N).
    \end{multline}
\end{theorem}
\begin{proof}
    Since $\cH$ is an RKHS with reproducing kernel $\xi(\cdot,\cdot)$, we have $\xi(x, t) = {\langle \xi_x, \xi_t \rangle}_{\cH}$, where $\xi_t(x) = \xi(x, t)$. This implies $\xi_t \in \cH \ (t \in \cX)$. However, since $\{ \varphi_j \}_{j=0}^{\infty}$ is complete and orthonormal in $\cH$, we have
    $$
       \left\| \xi_t - \sum_{j=0}^{\infty} {\langle \xi_t, \varphi_j \rangle}_{\cH} \varphi_j \right\|_{\cH} = 0,
    $$
    where ${\| f \|}_{\cH} := \sqrt{{\langle f, f \rangle}_{\cH}}$ $ (f \in \cH)$. From this, by Theorem~\ref{thm:pointwise} and the continuity of the absolute value function, we have
    $$
        \lim_{D \to \infty} \left| \xi_t(x) - \sum_{j=0}^{D-1} {\langle \xi_t, \varphi_j \rangle}_{\cH} \varphi_j(x) \right| = 0 \quad (x \in \cX).
    $$
    Consider the projection $\cP_D : \cH \to \myspan \{ \varphi_0, \ldots, \varphi_{D-1} \} \subset \cH$ $(D \in \N)$ defined as
    \begin{equation}
        \label{eq:Pd}
        \cP_D f := \sum_{j=0}^{D-1} {\langle f, \varphi_j \rangle}_{\cH} \varphi_j \quad (f \in \cH).
    \end{equation}
    Using this, we can define the error maps 
    $$
    e_{D, t} := \xi_t - \cP_D \xi_t = \sum_{j=D}^{\infty} {\langle \xi_t, \varphi_j \rangle}_{\cH} \varphi_j \quad (t \in \cX).
    $$
    By the orthonormality of $\{ \varphi_j \}_{j=0}^{\infty}$ we obtain
    $$
        \langle f, e_{D,t} \rangle_{\cH} = 0 \quad (f \in \myspan\{ \varphi_0,\ldots,\varphi_{D-1} \} \subset \cH).
    $$
    As a consequence we have
    \begin{equation}
        \label{eq:Riesz}
        {\langle \cP_D \xi_t, e_{D,x} \rangle}_{\cH} = 0 \quad (t, x \in \cX).
    \end{equation}
    From this, we obtain
    \begin{align}
        \label{eq:xierr}
        \begin{split}
        \xi(x, t) &= {\langle \xi_x, \xi_t \rangle}_{\cH} \\ &=  {\langle \cP_D \xi_x+ e_{D,x}, \cP_D \xi_t + e_{D,t} \rangle}_{\cH} \\ &= {\langle \cP_D \xi_x, \cP_D \xi_t \rangle}_{\cH} + {\langle \cP_D \xi_x, e_{D,t} \rangle}_{\cH} \\ & \qquad + \langle e_{D,x},\cP_D \xi_t  \rangle_{\cH}
        + \langle e_{D,x}, e_{D,t} \rangle_{\cH} \\ &\stackrel{\eqref{eq:Riesz}}{=} 
        {\langle \cP_D \xi_x, \cP_D \xi_t \rangle}_{\cH} + {\langle e_{D,x}, e_{D,t} \rangle}_{\cH}.
     \end{split}
    \end{align}
    Noting that $\{\varphi_j\}_{j=0}^{\infty}$ is an orthonormal system, we have
    \begin{equation}
        \label{eq:ons}
        {\langle \varphi_j, \varphi_n \rangle}_{\cH} = \delta_{jn} \quad (j,n \in \N \cup \{0\}).
    \end{equation}
    Thus, 
    \begin{align}
        \label{eq:xiD}
        \begin{split}
        \left\langle \cP_D \xi_x, \cP_D \xi_t \right\rangle_{\cH} &\stackrel{\eqref{eq:Pd}}{=} \left\langle \sum_{j=0}^{D-1} {\langle \xi_x, \varphi_j \rangle}_{\cH} \varphi_j,  \sum_{n=0}^{D-1} {\langle \xi_t, \varphi_n \rangle}_{\cH} \varphi_n \right\rangle_{\cH} \\ &=
        \sum_{j,n=0}^{D-1}       {\langle \xi_x, \varphi_j \rangle}_{\cH} \overline{ {\langle \xi_t, \varphi_n \rangle}_{\cH} } {\langle \varphi_j, \varphi_n \rangle}_{\cH} \\ &\stackrel{\textrm{Thm.~\ref{thm:kerProps} c)}}{=} 
        \sum_{j,n=0}^{D-1} \varphi_j(x) \overline{\varphi_n(t)} {\langle \varphi_j, \varphi_n \rangle}_{\cH} \\ &\stackrel{\eqref{eq:ons}}{=} \sum_{j=0}^{D-1} \varphi_j(x) \overline{\varphi_j(t)} = \xi_D(x, t).
      \end{split}
    \end{align}
    Eq.~\eqref{eq:xierr} and Eq.~\eqref{eq:xiD} give
    \begin{equation}
        \label{eq:errInner}
        |\xi(x,t) - \xi_D(x, t)| = \left|{\langle e_{D, x}, e_{D, t} \rangle}_{\cH}\right| \quad (x,t \in \cX).
    \end{equation}
    Note that since $\{ \varphi_j \}_{j=0}^{\infty}$ is complete and orthonormal in $\cH$ and by considering Theorem~\ref{thm:pointwise} and Theorem~\ref{thm:kerProps} c), for any $t \in \cX$ we have
    $$
        \xi(x, t) = \xi_t(x) = \sum_{j=0}^{\infty} \varphi_j(x) \overline{\varphi_j(t)}.
    $$
    Using this and Eq.~\eqref{eq:Pd} the statement follows from applying the Cauchy-Schwarz inequality to the right side of Eq.~\eqref{eq:errInner}.
\end{proof}

\begin{remarks}
\label{rem:errest}
\begin{enumerate}[label = \alph*)]
    \item In practice, $\cX$ is often assumed to be bounded, with $\{\varphi_j\}_{j=0}^{\infty}$ consisting of bounded, continuous functions. Then, Eq.~\eqref{eq:errest} can be written as
    \begin{equation*}
        \left| \xi(x, t) - \xi_D(x, t) \right| \leq \sup_{z \in \cX} \sum_{j=D}^{\infty} \left| \varphi_j(z) \right|^2 \quad (x, t \in \cX).
    \end{equation*}
    For concrete examples, this expression can sometimes be written in closed form  (see subsection~\ref{subsec:kernAprEx}).
    \item A consequence of Theorem~\ref{thm:kernelErr} is that for a fixed $D$, the error of the kernel approximation depends on the properties of $\{\varphi_j\}_{j=0}^{\infty}$. That is, choosing a different basis to express the truncated kernel from Eq.~\eqref{eq:truncKer} can yield vastly different approximations to a fixed kernel $\xi(\cdot,\cdot)$ (for the same $D$). Previous approaches to approximating kernels (see, e.g.,~\cite{rahimi, NEURIPS2018_4e5046fc}) advocate for using a fixed set of basis functions (e.g., in the case of Random Fourier Features, the basis is defined by trigonometric functions). By Theorem~\ref{thm:kernelErr}, this choice should also depend on the behavior of the kernel $\xi(\cdot,\cdot)$. Consequently, we can consider the following problem. Consider a parametrization $\Lambda \mapsto \{ \varphi_j^{\Lambda} \}_{j=0}^{\infty}$, where the function system is complete and orthonormal in a corresponding RKHS $\cH(\Lambda)$. Fixing $D \in \N$, we would like to find $\Lambda \in \cL$, such that the error in Theorem~\ref{thm:kernelErr} is minimal.
\end{enumerate}
\end{remarks}

\subsection{An example in $H_2(\D)$}
\label{subsec:kernAprEx}

We demonstrate the benefits of carefully choosing the orthonormal system $\{\varphi_j\}_{j=0}^{\infty} \subset \cH$  to construct the truncated kernel in Eq.~\eqref{eq:truncKer} through a concrete example. More precisely, we show how a parameterized family of orthonormal functions $\{ \varphi_j^{\Lambda} \}_{j=0}^{D-1} \subset \cH(\Lambda) \ (\Lambda \in \cL)$ can be used to approximate the kernel $\xi(\cdot,\cdot)$ of an infinite-dimensional RKHS $\cH$. Our task will be to find $\Lambda \in \cL$ for which we can obtain a minimal upper bound on $ \sup_{x, t \in X}  |\xi_D^{\Lambda}(x,t) - \xi(x,t)|$ for some $X \subset \cX$. In order to construct these $\Lambda$-dependent upper bounds, we use Theorem~\ref{thm:kernelErr}. We then show how the constructed error formula can be used to find the optimal choice for $\Lambda \in \cL$ to minimize the pointwise error bounds of the kernel approximation. We emphasize that, while we detail here a concrete example, using these steps, other $\Lambda \mapsto \{ \varphi_j^{\Lambda} \}_{j=0}^{D-1}$ parameterizations can be used to approximate the kernel $\xi(\cdot,\cdot)$ of an arbitrary infinite-dimensional RKHS $\cH$. The exact steps to repeat this process, however, always depend on the properties of the function system $\{ \varphi_j^{\Lambda} \}_{j=0}^{D-1}$.

Let $\cX$ be the open unit disk $\D := \{ z \in \C : |z| < 1 \}$ and denote its boundary by $\T := \{ z \in \C : |z| = 1 \}$. We shall also use the notation $\overline{\D} := \D \cup \T$. Denote by $\cA(\D)$ the disk algebra, that is, the set of complex analytic functions on $\D$. The Hardy-Hilbert space $H_2(\D)$ is defined as
\begin{multline}
    \nonumber
    H_{2}(\mathbb{D}) := \\ \left\{ f \in \mathcal{A}(\mathbb{D}) : \ \sup_{r < 1} \left( \frac{1}{2 \pi} \int_{-\pi}^{\pi} \left| f(r \eu^{\iu t}) \right|^2 \,\du t \right)^{1/2} < \infty \right\}.
\end{multline}
It is well-known~\cite{bprodbook}, that $H_2(\D)$ is an RKHS with the reproducing kernel
\begin{equation}
    \label{eq:CauchyKer}
    \xi(z, t) := \frac{1}{1 - \overline{t}z} \quad (z, t \in \D).
\end{equation}
The $H_2(\D)$ kernel $\xi(\cdot,\cdot)$ from Eq.~\eqref{eq:CauchyKer} is known as the Cauchy kernel~\cite{bprodbook}. 

The Takenaka-Malmquist (TM) functions~\cite{bprodbook} are defined as
\begin{multline}
    \label{eq:mtsys}
   \varphi_j^{\ba}(z) := \frac{\sqrt{1 - |a_j|^2}}{1 - \overline{a_j}z} \cdot \prod_{\nu=0}^{j-1} \frac{z - a_\nu}{1 - \overline{a_\nu}z} \\ = \frac{\sqrt{1 - |a_\nu|^2}}{1 - \overline{a_\nu}z} \cdot B_j^{\ba}(z), 
\end{multline}
where $\ba := (a_j)_{j=0}^{\infty}$ is a sequence in $\D$, $z \in \overline{\D} = \D \cup \T$, and \begin{equation}
\label{eq:bprod}
B_j^{\ba}(z) = \prod_{\nu=0}^{j-1} \frac{z - a_\nu}{1 - \overline{a_\nu}z}    
\end{equation}
 is a $j$-term monic Blaschke product~\cite{bprodbook}. These functions are orthonormal with respect to the $H_2(\D)$ inner product
$$
{\langle f,g \rangle}_{H_{2}(\mathbb{D})} := \frac{1}{2 \pi} \int_{-\pi}^{\pi} f(\eu^{\iu t})\overline{g(\eu^{\iu t})} \,\du t,
$$
that is ${\langle \varphi_j^{\ba}, \varphi_n^{\ba} \rangle}_{H_{2}(\mathbb{D})} = \delta_{jn}$ $(j,n = 0,1,\ldots)$.
 Blaschke products are self-maps on $\D$ and $\T$, and are inner functions of $H_2(\D)$ that play an important part in the factorization of Hardy spaces. If the generating sequence $\ba$ satisfies the Sz\'asz-Blaschke condition
\begin{equation}
    \label{eq:szasz}
    \sum_{j=0}^{\infty} (1 - |a_j|) = \infty,
\end{equation}
then the TM functions form a complete system in $H_2(\D)$. For a deep discussion on Blaschke products, TM functions, and corresponding topics, we recommend~\cite{bprodbook}. 

In the remainder of this section, we assume that the generating sequence $\ba$ is constant, that is $a_j = a \in \D$ for all $j \in \N$. Then the condition in Eq.~\eqref{eq:szasz} is satisfied and the corresponding TM functions are known as the discrete Laguerre system with
\begin{equation}
    \label{eq:dicLag}
    \varphi_j^{\ba}(z) = L_j^a(z) := \frac{\sqrt{1 - |a|^2}}{1-\overline{a}z} \left( \frac{z - a}{1 - \overline{a}z} \right)^j. 
\end{equation}
By the above, the discrete Laguerre system (for any choice of $a \in \D$) forms a complete orthonormal system in $H_2(\D)$. 

We are now ready to give an example for the approximation of the Cauchy kernel in Eq.~\eqref{eq:CauchyKer}. Consider an arbitrary parameter $a \in \D =: \cL$ and denote the corresponding discrete Laguerre function system by $\{ L_j^{a} \}_{j=0}^{\infty}$. Thus, in our example, we use the parameterization $a := \Lambda \mapsto \{ \varphi_{j}^{\Lambda} \}_{j=0}^{D-1} =: \{ L_{j}^{a} \}_{j=0}^{D-1}$. The $D$-dimensional subspaces $\cH(\Lambda) = \cH(a) := \myspan \{ L_k^a \,:\, j=0,\ldots,D-1 \}$ are also known as model spaces~\cite{bprodbook}. We fix $D \in \N$ and consider the approximation
\begin{equation*}
        \xi_D^a(z, t) := \sum_{j=0}^{D-1} L_j^a(z) \overline{L_j^a(t)} \quad (a \in \D).
\end{equation*}
We are interested in estimating $\sup_{z, t \in X}|\xi(z,t) - \xi_D^a(z,t)| \ (X \subset \D)$ for two reasons. First, this estimate will verify that for a fixed $D$, the approximation error depends on the parameter $a \in \D$ and thus selecting different parameters can improve the approximation. Second, it helps us to illustrate how the error estimate in Theorem~\ref{thm:kernelErr} can be used to select an optimal parameter $a$ (and hence Laguerre basis) to approximate the kernel of a given RKHS. 

By Theorem~\ref{thm:kernelErr}, we have 
\begin{multline}
    \label{eq:LagErr}
    \left| \xi(z,t) - \xi_D^a(z, t) \right| = \left| \frac{1}{1 - \overline{t}z} - \sum_{j=0}^{D-1} L_j^a(z) \overline{L_j^a(t)} \right| \\ \leq \left( \sum_{j=D}^{\infty} \left|L_j^a(z)\right|^2 \right)^{1/2} \left( \sum_{j=D}^{\infty} \left|L_j^a(t)\right|^2 \right)^{1/2}.
\end{multline}
    
To obtain a good expression for the right hand side of Eq.~\eqref{eq:LagErr}, first notice that
\begin{multline}
    \nonumber
    \sum_{j=D}^{\infty} \left|L_j^a(z)\right|^2 \stackrel{~\eqref{eq:dicLag}}{=} \sum_{j=D}^{\infty} \left| \frac{\sqrt{1 - |a|^2}}{1 - \overline{a}z}\right|^2 \left|B^a(z)\right|^{2j} \\ = 
    \frac{1 - |a|^2}{|1 - \overline{a}z|^2} \sum_{j=D}^{\infty} \left|B^a(z)\right|^{2j},
\end{multline}
where $B^a(z) := (z - a)/(1 - \overline{a}z)$ is a Blaschke factor as used in Eq.~\eqref{eq:bprod}. Assuming $z, a \in \D$ and using the fact that $B^a$ is a self-map on $\D$, we have
\begin{multline}
    \label{eq:almostSimple}
    \frac{1 - |a|^2}{|1 - \overline{a}z|^2} \sum_{j=D}^{\infty} \left|B^a(z)\right|^{2j} \\ = \frac{1 - |a|^2}{|1 - \overline{a}z|^2} \left|B^a(z)\right|^{2D} \cdot \frac{1}{1 - \left|B^a(z)\right|^2} \\ =
    \frac{1 - |a|^2}{|1 - \overline{a}z|^2 (1 - |B^a(z)|^2)} \cdot \left|B^a(z)\right|^{2D}.
\end{multline}
Substituting the identity (see, e.g.,~\cite[Lem.~2.2.1]{bprodbook})
$$
    1 - |B^a(z)|^2 = \frac{(1 - |z|^2)(1 - |a|^2)}{|1 - \overline{a}z|^2} \quad (a,z \in \D)
$$
into Eq.~\eqref{eq:almostSimple} yields
$$
    \sum_{j=D}^{\infty} \left|L_j^a(z)\right|^2 = \frac{|B^a(z)|^{2D}}{1 - |z|^2}.
$$
Since this holds for any $z \in \D$ we obtain
\begin{multline}
    \nonumber
    \left| \xi(z,t) - \xi_D^a(z, t) \right| \stackrel{\textrm{Remark~\ref{rem:errest} (a)}}{\leq}  \sup_{u \in \D} \frac{|B^a(u)|^{2D}}{1 - |u|^2} \\ =:
    \sup_{u \in \D} \eta^a_{D}(u).
\end{multline}
Next, we would like to obtain a uniform bound on $\eta^a_{D}$. Since $\lim_{|z| \to 1} \eta^a_{D}(z) = \infty$, we restrict our search to the domain
\begin{equation}
    \label{eq:reducedDisk}
    \D_\rho := \{ z \in \D_\rho \,:\, |z| \leq \rho < 1 \},
\end{equation}
where $0 < \rho < 1$ is fixed. That is, we choose the Laguerre parameter to satisfy $|a| < 1$ and look for a uniform bound (dependent only on $D$ and $a$) for $\eta^a_{D}$ over $\D_\rho$. We shall make use of the following identity.
\begin{lemma}
    \label{thm:BlaschEst}
    Given $a \in \D$ and $|z| \leq \rho < 1$, the Blaschke factor $B_a$ satisfies
    \begin{equation}
        |B^a(z)| = \frac{|z - a|}{|1 - \overline{a}z|} \leq \frac{|a| + \rho}{1 + |a| \rho}.
    \end{equation}
\end{lemma}
The proof of Lemma~\ref{thm:BlaschEst} can be found in~\cite[Lem.~2.2.1]{bprodbook}.
This identity yields
\begin{multline}
    \label{eq:unifBound}
    |\xi(z,t) - \xi_D^a(z,t) |\leq \sup_{u \in \D_\rho} \eta_{D}^{a}(u) \\ = \left(\frac{|a| + \rho}{1 + |a| \rho}\right)^{2D} \frac{1}{1 - \rho^2} =: E_D(|a|). 
\end{multline}
The uniform bound in Eq.~\eqref{eq:unifBound} allows us to achieve the first of our objectives. Indeed, since $E_D$ depends on $|a|$, for a fixed $D$ we can obtain different uniform error bounds of the kernel approximation which depend on our choice of the orthonormal basis $\{ L_j^a \}_{j=0}^{\infty}$.

Finally, noticing that the function $E_D$ from Eq.~\eqref{eq:unifBound} is increasing on the interval $[0, \rho]$, we can use Eq.~\eqref{eq:unifBound} to choose a Laguerre system which minimizes the error estimate. Indeed, choosing $a = 0$ in this case minimizes the bound. We obtain
\begin{multline}
    \nonumber
     \left|\xi(z,t) - \xi_{D}^{0} (z,t) \right| = \left| \frac{1}{1 - \overline{t}z} - \sum_{j=0}^{D-1} L_j^{0}(z) \overline{L_j^0(t)} \right| \\\stackrel{\eqref{eq:dicLag}}{=} 
    \left| \frac{1}{1 - \overline{t}z} - \sum_{j=0}^{D-1} z^j \overline{t}^j \right|  \leq \frac{\rho^{2D}}{1 - \rho^2} \quad (z, t \in \D_\rho).
\end{multline}
This satisfies our second objective. Our calculations serve as an example of how to apply Theorem~\ref{thm:kernelErr} to obtain good quality kernel approximations. We note that in contrast to previous works (see, e.g.,~\cite{rahimi}), our error estimates are deterministic in nature and depend only on the properties of the considered RKHS and the parametrization $\Lambda \mapsto \cH(\Lambda)$.

This example also serves as motivation for our results in section~\ref{sec:adaptker}. So far, we have discussed adaptive kernel methods, where the RKHS $\cH(\Lambda) \ (\Lambda \in \cL)$ from Eq.~\eqref{eq:varproj} is chosen so that its kernel can be expressed as $\xi_D^{\Lambda}(x, t) = {\langle \Phi(x), \Phi(t) \rangle}_{\C^D} \ (x,t \in \cX)$ for some reasonable $D \in \N$ and $\xi_D^{\Lambda}(\cdot,\cdot) \approx \xi(\cdot,\cdot)$, where $\xi(\cdot,\cdot)$ is the reproducing kernel of a fixed RKHS $\cH$. This idea is also used by previous methods such as Random Fourier Features (RFF)~\cite{rahimi}. Unfortunately, this approach can be suboptimal when learning a map $F \in \cF$. To demonstrate this, consider 
\begin{equation}
    \label{eq:exampleF}
    F(z) = \frac{1}{1 - \overline{\lambda}z},
\end{equation}
where $z \in \overline{\D}$ and $\lambda \in \D \setminus \{ 0 \}$. Then, $F \in \myspan\{ L_0^{\lambda}\} =: \cH(\lambda) \subset H_2(\D) \ (\lambda \in \cL =: \D)$. That is, $F$ belongs to a one-dimensional subspace of $H_2(\mathbb{D})$ spanned by the first Laguerre function (see Eq.~\eqref{eq:dicLag}). Hence, if we choose $\Lambda := \lambda$ from  Eq.~\eqref{eq:exampleF} in the parametrization $\Lambda \mapsto \cH(\Lambda)$, the best approximating kernel model can completely recover $F$. Using the logic of RFF-like methods, we should approximate $F$ with a kernel model that belongs to a $D$-dimensional RKHS whose kernel approximates the $H_2(\D)$ kernel (see Eq.~\eqref{eq:CauchyKer}) well. For this example, choose $D=1$. By Eq.~\eqref{eq:unifBound}, choosing $\Lambda = 0$ provides the best uniform error bound between the kernels of the model space $\myspan\{ L_0^{0} \} =: \cH(0)$ and $H_2(\D)$. Denote the best approximation to $F$ in $\cH(0)$ by $f^{(0)}$ and in $\cH(\lambda)$ by $f^{(\lambda)}$, where $\lambda$ is the parameter from Eq.~\eqref{eq:exampleF}. Since $F \not\in \myspan\{ L_0^{0} \}$, we have $\big\| F - f^{(0)} \big\|_{H_2(\D)} > 0$. On the other hand, $\big\| F - f^{(\lambda)}\big\|_{H_2(\D)} = 0$. Thus, a kernel model from $\cH(\lambda)$ provides a better approximation to $F$ than the kernel model from $\cH(0)$ even though the kernel of the latter space provides a better approximation (or at least a tighter error bound) to Eq.~\eqref{eq:CauchyKer}, than the kernel of $\cH(\lambda)$. 

Using this observation, we deviate from the traditional RFF viewpoint of obtaining kernel models in low-dimensional RKHSs whose kernels provide good approximations to kernels of infinite-dimensional RKHSs. Instead, in section~\ref{sec:adaptker}, we propose kernel models $f \in \cH(\Lambda)$ that provide a \textit{good approximation to the map-to-be-learned} $F \in \cF$, regardless of the properties of the kernel that induces the low-dimensional RKHS $\cH(\Lambda)$.

\section{Adaptive kernel models in finite-dimensional RKHSs}
\label{sec:adaptker}

In this section, we consider the proposed adaptive kernel methods described by the variable projection operator in Eq.~\eqref{eq:varproj}. We define adaptive kernel methods whose image space consists of finite-dimensional RKHSs and discuss corresponding kernel models. First, we discuss some properties of these scalable (see Eq.~\eqref{eq:linKer}) adaptive methods, then we investigate the training problem in Eq.~\eqref{eq:optimProb}. 

\subsection{Finite-dimensional adaptive kernel methods}
\label{subsec:orthogonal}

We begin by formally defining finite-dimensional adaptive kernel methods and corresponding kernel models.

\begin{definition}[Finite-dimensional adaptive kernel methods/models]
\label{def:fdakm}
    Let $F \in \cF$, $\cS$ denote the set of all possible datasets $\{(t_k, y_k) \}_{k=1}^q \subset \cX \times \cY \ (q \in \N)$, and $\cE \subseteq \{ E : \left(\cX \times \cY \times \cY\right)^q \to [0, \infty) \}$ be the set of admissible error functionals. Let $S \in \cS$ and $E \in \cE$. Finally, let $\cL \subseteq \C^p$ be a nonempty parameter domain and suppose that for any $\Lambda \in \cL$ and fixed $D \in \N$, the parametrization $\Lambda \mapsto \{ \varphi_j^{\Lambda} \}_{j=0}^{D-1}$ is well-defined, where the $\cX \to \C$ maps $\{\varphi_j^{\Lambda}\}_{j=0}^{D-1}$ form a linearly independent system in a corresponding $D$-dimensional Hilbert space $\cH(\Lambda)$. We call the projection $\cP_{S, E}^{\Lambda} : \cF \to \cH(\Lambda)$ defined in \eqref{eq:varproj} and represented as
    \begin{multline}
        \label{eq:adaptiveKernelModels}
        \cP_{S, E}^{\Lambda} F := \sum_{k=1}^{q} c_k \left\langle \Phi^{\Lambda}(\cdot), \Phi^{\Lambda}(t_k) \right\rangle_{\C^{D}} \\ = \left\langle \Phi^{\Lambda}(\cdot), \overline{w} \right\rangle_{\C^{D}}  = \sum_{j=0}^{D-1} w_j \varphi_j^{\Lambda}(\cdot)
    \end{multline}
    a \emph{finite-dimensional adaptive kernel method} and the right-hand side a \emph{finite-dimensional adaptive kernel model}. In Eq.~\eqref{eq:adaptiveKernelModels}, the parameter-dependent feature map $\Phi^{\Lambda} : \cX \to \C^D$ is given by
    \begin{equation}\label{eq:featuremap}
        \Phi^{\Lambda}(\cdot) := \left[ \varphi_0^{\Lambda}(\cdot), \ldots, \varphi_{D-1}^{\Lambda}(\cdot) \right]^\top.
    \end{equation}
    
\end{definition}

\begin{remarks}
\label{rem:fdakm}
    \begin{enumerate}
        \item By the representer theorem (Theorem~\ref{thm:representer}, see also~\cite{representer}), the result of all kernel methods belongs to a $q$-dimensional RKHS. By Eq.~\eqref{eq:adaptiveKernelModels} however, the proposed adaptive kernel models belong to a $D$-dimensional RKHS (usually with $D \ll q$) defined by the parameterization $\Lambda \mapsto \{\varphi_j^{\Lambda} \}_{j=0}^{D-1} \ (\Lambda \in \cL)$. This makes the proposed methods usable with large datasets (see also subsection~\ref{subsec:benchmark}). 
        
        \item Previous methods improved the kernel model (with respect to $E$) by increasing $D$ (see, e.g.,~\cite{rahimi}). When training adaptive kernel methods however, we can modify $\Lambda$ to improve the quality of the models while retaining model complexity. 
    \end{enumerate}
\end{remarks}

Next, we discuss some properties of the proposed adaptive kernels inducing $D$-dimensional RKHSs.
\begin{theorem}[Adaptive kernels of finite dimensional RKHSs]
\label{thm:adaptLinKer}
    Let $\cL \subseteq \C^p$ be a parameter domain.  For $\Lambda\in \cL$ consider the parametrization $\Lambda \mapsto \{ \varphi_j^{\Lambda} \}_{j=0}^{D-1}$ and the kernel
    \begin{equation}
        \label{eq:adaptfinker}
        \xi_D^{\Lambda}(x, t) := {\langle \Phi^{\Lambda}(x), \Phi^{\Lambda}(t) \rangle}_{\C^D} \quad (x, t \in \cX),
    \end{equation}
    where the parameterization and the feature map $\Phi^{\Lambda}$ are defined according to Eq.~\eqref{eq:featuremap}. Assume that $\{\varphi_j^{\Lambda} \}_{j=0}^{D-1}$ are linearly independent and span the $D$-dimensional RKHS $\cH(\Lambda)$. Then $\xi_D^{\Lambda}(\cdot,\cdot)$ is a Hermitian, positive definite reproducing kernel in $\cH(\Lambda)$. 
\end{theorem}
\begin{proof}
    Hermitian symmetry and positive definiteness follow immediately from Eq.~\eqref{eq:adaptfinker}. We show the reproducing property. First, define $\cH(\Lambda)$ as
    \begin{equation}
        \label{eq:HLambda}
        \cH(\Lambda) := \left\{ g_w \in \cF \,: \, g_w(\cdot) := \overline{w}^{\ast} \Phi^{\Lambda}(\cdot) = \sum_{j=0}^{D-1} w_k \varphi_j^{\Lambda}(\cdot) \right\}.
    \end{equation}
    The inner product corresponding to $\cH(\Lambda)$ is given as ${\langle g_{w_1}, g_{w_2} \rangle}_{\cH} := {\langle w_1, w_2 \rangle}_{\C^{D}} \ (w_1, w_2 \in \C^D)$. Since we have $\xi_D^{\Lambda}(\cdot, t) = {\langle \Phi^{\Lambda}(\cdot), \Phi^{\Lambda}(t)  \rangle}_{\C^D} = \Phi^{\Lambda}(t)^{\ast} \Phi^{\Lambda}(\cdot)$, we immediately obtain the reproducing property, i.e.,
    \begin{multline*}
        {\langle g_w, \xi_D^{\Lambda}(\cdot, t) \rangle}_{\cH}  = \left\langle w,\overline{\Phi^{\Lambda}(t)}\right\rangle_{\C^D} \\ = \left\langle \Phi^{\Lambda}(t), \overline{w}\right\rangle_{\C^D} = \overline{w}^{\ast} \Phi^{\Lambda}(t) \stackrel{\eqref{eq:HLambda}}{=} g_w(t).
    \end{multline*}
\end{proof}
\begin{remarks}
\label{rem:adaptOrt}
    \begin{enumerate}
        \item Suppose $F \in L_2(\cX)$ and that for any choice of $\Lambda \in \cL$, we obtain an orthonormal family $\{ \varphi_j^{\Lambda} \}_{j=0}^{D-1}$ (with respect to the $L_2(\cX)$ inner product). Then, the induced $D$-dimensional Hilbert space $\cH(\Lambda)$ will inherit the inner product of $L_2(\cX)$. 
        \item There are many well-known parameterized (often orthonormal) function families $\{ \varphi_j^{\Lambda} \}_{j=0}^{D-1}$ which can be used to construct reproducing kernels of the form presented in Theorem~\ref{thm:adaptLinKer}. Examples include the TM functions (see Eq.~\eqref{eq:mtsys} and~\cite{bprodbook}) for the RKHS $H_2(\D)$ and its finite-dimensional subspaces. 
        Other examples include kernels constructed from adaptive and weighted adaptive Hermite functions~\cite{weightedHerm}. These induce RKHSs which are finite-dimensional subspaces of the Hilbert space $L_2\big(\R^N\big)$ (which is not an RKHS itself).
        \item If there is some a priori information about $F$, it can be used to choose an appropriate family $\{ \varphi_j^{\Lambda} \}_{j=0}^{D-1}$. In such cases, the proposed adaptive kernel models can contain physically meaningful information in the parameter $\Lambda$.
        In this regard, the proposed adaptive kernel methods can be used to build physics-informed ML models (see, e.g.,~\cite{dovsilovic2018explainable}).
    \end{enumerate}
\end{remarks}

\subsection{Training finite-dimensional adaptive kernel models}
\label{subsec:GVP}

In this section, we consider the problem of finding the best finite-dimensional adaptive kernel model (see Definition~\ref{def:fdakm}) for a given map-to-be-learned $F \in \cF$, dataset $S \in \cS$ and loss function $E \in \cE$. We uphold all assumptions of Definition~\ref{def:fdakm} and Theorem~\ref{thm:adaptLinKer}. The training problem for a general adaptive kernel model is motivated by Eq.~\eqref{eq:optimProb} and is given by
\begin{equation}
 \label{eq:iteredmin}
 \min_{\Lambda \in \cL} \min_{g \in \cH(\Lambda)} E((t_1, g(t_1), y_1), \ldots, (t_q, g(t_q), y_q)),
\end{equation}
where $y_k = F(t_k)$ $(k=1,\ldots,q)$.
Considering that an optimal kernel model $f^{\Lambda^\ast}$ is an adaptive kernel model from Eq.~\eqref{eq:adaptiveKernelModels}, it belongs to a $D$-dimensional RKHS. Furthermore it can be written as
$$
    f^{\Lambda^\ast} = \sum_{j=0}^{D-1} w_j \varphi_j^{\Lambda^\ast}.
$$
The existence of such a minimizer is guaranteed by the following.
\begin{theorem}[Existence of a minimizer]
    \label{thm:exist}
    Let $\cL \subset \C^p$ be a  nonempty compact set and let $\{ (t_k, y_k) \}_{k=1}^{q} \subset \cX \times \cY$, where $\cX \subseteq \C^N$ and $\cY \subseteq \C$. Let $D \in \N$ and suppose that for any $\Lambda \in \cL$, $\Lambda \mapsto \myspan\{\varphi_0^{\Lambda}, \varphi_1^{\Lambda}, \ldots, \varphi_{D-1}^{\Lambda} \} =: \cH(\Lambda)$ is a $D$-dimensional RKHS of functions $\cX \to \cY$. Moreover, suppose that the maps $\Lambda \mapsto \varphi_j^{\Lambda}$ are continuous (with respect to the norm induced by $\cH(\Lambda)$ in the co-domain) for $j=0,\ldots,D-1$.
    Let $w:=\left[ w_0, \ldots, w_{D-1} \right]^\top\in \C^D$ and
    \begin{equation*}
     f^{\Lambda}(w;\cdot) := \sum_{j=0}^{D-1} w_j \varphi_j^{\Lambda}(\cdot) \quad (\Lambda \in \cL).
    \end{equation*}
    Suppose that the error functional $E$ in \eqref{eq:iteredmin} is continuous with respect to each of the second variables of each triple $(t_k,g(t_k),y_k)$ $(k=1,\ldots,q)$.
    Finally, assume that there exists some  large enough $M > 0$, such that the continuous compact-valued correspondence $C : \cL \rightrightarrows \C^D$, where
    \begin{equation}
        \label{eq:sublev}
        C(\Lambda) := \{ w \in \C^D \,:\, \tE(\Lambda, w) \leq M \}
    \end{equation}
    is non-empty for all $\Lambda \in \cL$, where 
    \begin{equation}
        \label{eq:jointE}
        \tE(\Lambda, w) := E((t_1, f^{\Lambda}(w; t_1), y_1), \ldots, (t_q, f^{\Lambda}(w; t_q), y_q).
    \end{equation}
    
    Then, the optimization problem in Eq.~\eqref{eq:iteredmin} has a solution.
\end{theorem}
\begin{proof}
    Notice that since $\Lambda \mapsto \varphi_j^{\Lambda}$ is continuous for $j=0,\ldots,D-1$, the map
    $$
        (\Lambda,w) \mapsto  f^{\Lambda}(w;\cdot) 
    $$
    is continuous. Hence, $\tE$ in Eq.~\eqref{eq:jointE} is also continuous. Because of this and since $C(\Lambda)$ is non-empty and compact for each $\Lambda \in \cL$,
    the extreme value theorem implies that $\tE(\Lambda, \cdot)$ attains a minimum for fixed $\Lambda \in \cL$. Because of this, the sets
    \begin{multline}
        \label{eq:argminSets}
        C^{\ast} (\Lambda) := \argmin_{w \in \C^D} \tE(\Lambda, \cdot) \\ = \left\{ w^* \in \C^D : \tE(\Lambda, w^*) \le \tE(\Lambda, w) \text{ for all } w \in \C^D \right\}
    \end{multline}
    are nonempty and compact for each $\Lambda \in \cL$. Define the function
    \begin{equation}
        \label{eq:cE}
        \cE(\Lambda) := \min_{w \in \C^D} \tE(\Lambda, w).
    \end{equation}
    Due to the continuity of $\tE$ in both variables and the compactness of the sets $C^{\ast}(\Lambda)$ from Eq.~\eqref{eq:argminSets}, we can apply Berge's maximum theorem~\cite[Chap.~E, Sect.~3]{ok2007real}. As a consequence, $\cE$ from Eq.~\eqref{eq:cE} is continuous. However, since $\cL$ is compact, it necessarily attains a minimum.
\end{proof}
\begin{remarks}
    \begin{enumerate}
        \item Uniqueness of the solution can be guaranteed by making strict convexity assumptions for $E$ or adding an appropriate regularization term. 
        \item For many concrete choices of the parameterized family $\{ \varphi_j^{\Lambda} \}_{j=0}^{D-1}$, we cannot assume the compactness of $\cL$. Nevertheless, a solution can still exist. We refer to~\cite{weightedHerm} for such an example, where $\cL$ is not compact, $\{\varphi_j^{\Lambda}\}_{j=0}^{D-1}$ are chosen as weighted Hermite functions, and $E$ is chosen as the least-squares error.
        \item The intuition behind the assumption on the continuity of $C : \cL \rightrightarrows \C^D$ is that the parametrization should be such that the sub-level sets of $\tE$ in Eq.~\eqref{eq:jointE} do not change  instantaneously as $\Lambda$ changes.
        \item The continuity of the correspondence  $C : \cL \rightrightarrows \C^D$ can be challenging to verify. However, since in practical cases we have $\cL \subset \C^p$ and $C$ is a correspondence between metric spaces. Then, by~\cite[Chap.~2, Prop.~ 5]{ok2007real}, the continuity of $C$ in Eq.~\eqref{eq:sublev} is equivalent to its continuity in the Hausdorff metric, which can be easier to verify. 
    \end{enumerate}
    \label{rem:existremarks}
\end{remarks}

\subsection{Connection to Random Fourier Features methods}
\label{subsec:prevMeth}


We proceed to show, how a popular kernel method known as Random Fourier Features (RFF) proposed by Rahimi and Recht in~\cite{rahimi} can be interpreted as a special case of the adaptive reproducing kernel models proposed in Eq.~\eqref{eq:adaptiveKernelModels}.

We begin with a short summary of the RFF method. Consider an infinite-dimensional RKHS $\cH$ of functions $\cX \to \cY$ ($\cX \subseteq \R^N, \cY \subseteq \R$) with a reproducing kernel $\xi : \cX \times \cX \to \cY$. If $\xi(\cdot,\cdot)$ is also continuous and shift invariant, then Bochner's theorem~\cite{rudin2017fourier} can be applied to it.
\begin{theorem}[Bochner's theorem]
    A continuous kernel $\xi(x, t) := \zeta(x-t) \ (x,t \in \R^N)$ is positive definite if and only if $\zeta$ is the Fourier transform of a positive finite Borel measure.
    
\end{theorem}
 This theorem is exploited in~\cite{rahimi} to obtain unbiased estimator of $\xi(\cdot,\cdot)$. More precisely, if $\zeta$ is properly scaled, its Fourier transform is a probability distribution, i.e., there exists a probability density function $p:\R^N \to [0,\infty)$ such that 
\begin{align}
    \label{eq:RFFker}
    \begin{split}
    \xi(x, t) &= \zeta(x-t) \\ &= \int_{\R^N} p(\lambda) \eu^{\iu \langle \lambda, x-t \rangle_{\R^N} } \,\du \lambda \\ &\approx  \sum_{j=0}^{D-1} \eu^{\iu \langle \lambda_j, x \rangle_{\R^N}} \eu^{-\iu \langle \lambda_j, t \rangle_{\R^N}} =: \xi_D^{\Lambda}(x, t)
\end{split}
\end{align}
for all $x,t \in \R^N$ and some parameters $\lambda_j \in \R^N$ $(j=0,\ldots,D-1)$.

When using the RFF modeling scheme, the parameter 
\begin{equation}
    \label{eq:RFFpars}
    \Lambda := [\lambda_0, \lambda_1, \ldots, \lambda_{D-1}] \in \cL := \R^{N \times D}
\end{equation}
is randomly selected according to the probability distribution $p$. Note that the right-hand side of Eq.~\eqref{eq:RFFker} coincides with the truncated kernel definition in Eq.~\eqref{eq:adaptfinker}. Indeed, provided that $\lambda_j \neq \lambda_n \ (j \neq n)$, the trigonometric basis functions from Eq.~\eqref{eq:RFFker} given by
\begin{equation}
    \label{eq:trigsys}
    \varphi_j^{\Lambda}
    := \eu^{\iu \langle \lambda_j, \cdot \rangle_{\R^N}} \quad (j=0,\ldots,D-1)
\end{equation}
span a $D$-dimensional subspace of the ambient Hilbert space $L_2(\cX)$, if $\cX \subset \R^N$ is compact. This subspace is an RKHS and the RFF model can be interpreted as a special case of Eq.~\eqref{eq:adaptiveKernelModels} with the feature transformation
$$
\Phi^{\Lambda}(\cdot) := \left[\eu^{\iu \langle \lambda_0, \cdot \rangle_{\R^N}}, \ldots, \eu^{\iu \langle \lambda_{D-1}, \cdot \rangle_{\R^N}}\right]^\top.
$$
Indeed,  using the kernel approximation $\xi_D^{\Lambda}$ from Eq.~\eqref{eq:RFFker}, we can define the RFF kernel model
\begin{align}
\label{eq:RFFmodel}
\begin{split}
    f^{\Lambda} &:= \sum_{k=1}^{q} c_k \xi_D^{\Lambda}(\cdot, t_k) \\ &= \sum_{k=1}^{q} c_k \sum_{j=0}^{D-1} \eu^{\iu \langle \lambda_j, \cdot \rangle_{\R^N}} \eu^{-\iu \langle \lambda_j, t_k \rangle_{\R^N}} \\ &= 
    \sum_{j=0}^{D-1} \eu^{\iu \langle \lambda_j, \cdot \rangle_{\R^N}} \sum_{k=1}^q \overline{c_k \eu^{\iu \langle \lambda_j, t_k \rangle_{\R^N}}} = \sum_{j=0}^{D-1} w_j \eu^{\iu \langle \lambda_j, \cdot \rangle_{\R^N}}.
 \end{split}
\end{align}
Thus, the RFF algorithm can be understood as a special case of the methods proposed here. In terms of our proposed framework, when using the RFF method, the nonlinear kernel parameters $\Lambda$ from Eq.~\eqref{eq:RFFpars} are optimized using a Monte Carlo approach by exploiting Bochner's theorem. In this way, probabilistic guarantees can be obtained for how well $\xi_D^{\Lambda}(\cdot,\cdot)$ approximates the kernel $\xi(\cdot,\cdot)$ of the ambient RKHS~\cite{rahimi}. Here we argue (and show through experiments in section~\ref{sec:exp}), that despite these guarantees on the quality of this approximation, the parameters $\Lambda$ obtained this way are suboptimal and usually do not solve the optimization problem in  Eq.~\eqref{eq:optimProb}. In fact, applying only a few iterations of a gradient-based optimization method to refine $\Lambda$ obtained by the RFF method could significantly improve the quality of the kernel model, when using the model from Eq.~\eqref{eq:RFFmodel} as an initial approximation. Additionally, in all our experiments, see section \ref{sec:exp}, the computational overhead due to the inclusion of $\Lambda$ as trainable parameters never causes a significant increase in training time. In fact, because we are looking for the solution in a family of RKHSs $\cH (\Lambda)$ instead of a single fixed one, stopping  criteria for training are usually met in fewer iterations.


Finally, assuming the map $\Lambda \mapsto \{\varphi_j^{\Lambda}\}_{j=0}^{D-1}$ along with the error functional $E$ are differentiable with respect to $\Lambda$,
Algorithm \ref{alg:adaptiveKernelGradient} is a prototype algorithm that illustrates how the gradient descent method can be employed for adaptive kernel model training.

\begin{algorithm}[t]
\caption{Gradient descent training of an adaptive kernel model}
\label{alg:adaptiveKernelGradient}
\begin{algorithmic}[1]
\Require Data $\{(t_k,y_k)\}_{k=1}^q$; order $D\in\mathbb{N}$;
mapping $\Lambda \mapsto \big\{\varphi_j^{\Lambda}\big\}_{j=0}^{D-1}$ (lin. indep., differentiable in $\Lambda$);
error functional $\tE$ as in Eq.~\eqref{eq:jointE};
step sizes $\eta_{\Lambda},\eta_w>0$; tolerance $\varepsilon>0$; max. iterations $I_{\max}\in\mathbb{N}$.
\Ensure Optimal parameters $(\Lambda^\star,w^\star)$.
\State  Choose initial parameters $\Lambda^{(0)}\in\mathcal{L}$, $w^{(0)}\in\mathbb{C}^D$.
\State $i \gets 0$.

\While{$i < I_{\max}$}
    \State  Build $\{\varphi_j^{\Lambda^{(i)}}\}_{j=0}^{D-1}$.
    \Comment{Optionally orthonormalize.}
    \State  Form the kernel model $ \cP_{S, E}^{\Lambda} F$. \Comment{Eq.~\eqref{eq:adaptiveKernelModels}.}
    \State Compute
    \[
        g_{\Lambda}^{(i)} \gets \nabla_{\Lambda} \tE\left(\Lambda^{(i)},w^{(i)}\right), \quad g_w^{(i)} \gets \nabla_w  \tE\left(\Lambda^{(i)},w^{(i)}\right).
    \]
    \State Perform a gradient step
    \[ 
        \Lambda^{(i+1)} \gets \Lambda^{(i)} - \eta_{\Lambda}\, g_{\Lambda}^{(i)}, \quad w^{(i+1)} \gets w^{(i)} - \eta_w\, g_w^{(i)}.
    \]

    \If{$\big\|g_{\Lambda}^{(i)}\big\| + \big\|g_w^{(i)}\big\| \le \varepsilon$}
        \State $i \gets i+1$.
        \State \textbf{break}
    \EndIf
    
    \State $i \gets i+1$.
\EndWhile

\State \Return $(\Lambda^\star,w^\star) \gets \left(\Lambda^{(i)},w^{(i)}\right)$.\Comment{Approx. sol. of Eq~\eqref{eq:optimProb}.}
\end{algorithmic}
\end{algorithm}
\begin{remarks} Instead of the gradient method used in Algorithm \ref{alg:adaptiveKernelGradient}, we can employ other optimization schemes. For example, the RFF method can be interpreted as an adaptive kernel method where the optimization is a stochastic sampling strategy.
\end{remarks}

\section{Experiments}
\label{sec:exp}

In this section, we provide some numerical experiments to illustrate the effectiveness of the adaptive kernel models introduced in section~\ref{sec:adaptker}. First, we consider the identification of a single-input single-output linear-time invariant (SISO LTI) dynamical system. This example is used to illustrate some important properties of the proposed approach. In particular, it shows that if a parametrization $\Lambda \mapsto \cH(\Lambda)$ can be obtained such that $F \in \cH(\Lambda^{\ast})$ for some $\Lambda^{\ast} \in \cL$, then the proposed adaptive methods can be used to completely recover $F$. Then, we compare the performance of adaptive kernel models to existing kernel approximations using benchmark classification and regression datasets. Our experiments are conducted on a workstation with an Intel(R) Xeon(R) W-2123 CPU @ 3.60GHz, 64 GB RAM, NVIDIA TITAN RTX GPU and Python 3.10.12 with Pytorch 2.4.0. All of the implemented models are trained and evaluated on the GPUs. We note that all of the results disclosed here are fully reproducible using our publicly available PyTorch implementation of the proposed methods (see the code availability statement at the end of the article).

\subsection{LTI system identification}
\label{subsec:lti}

We consider the problem of identifying a discrete-time SISO LTI system. Such systems can be used to model a variety of physical and engineering problems, we recommend~\cite{szidar, ljungsysid, vandenhof} for some examples. Let $\ell$ denote the vector space of (unilateral) complex sequences. The time domain representation of a discrete time SISO LTI system is given by the convolution equation
\begin{equation}
    \label{eq:timeDom}
    \by := \bh \ast \bu \quad (\by, \bh, \bu \in \ell),
\end{equation}
where $\bu$ and $\by$ are referred to as the input and output sequences, respectively, with $\bh$ denoting the impulse response. Here we implicitly assume the causality ($\bh_n = 0$ $(n < 0)$) and bounded-input bounded-output (BIBO) stability~\cite{szidar} of the system. In this case, the impulse response is absolutely summable, i.e., it satisfies $\bh \in \ell_1$. Furthermore, we assume that the impulse response can be written as
\begin{equation}
    \label{eq:impresp}
    h_n := \sum_{j=0}^{D-1} w_j \overline{\lambda_j}^n \quad (n \in \N)
\end{equation}
for some $D \in \N$ and $w_j \in \C$, $\lambda_j \in \D$ $(k=0,\ldots,D-1)$. 
The above assumptions are often made for systems in real applications. By Eq.~\eqref{eq:impresp}, gaining access to the parameters $\lambda_j$ and $w_j \ (j=0,\ldots,D-1)$, allows us to describe the system's behavior completely. Consider the $\cZ$-transform defined as
\begin{equation}
    \label{eq:Ztrans}
    \cZ[\bh](z) := \sum_{n=0}^{\infty} h_n z^n \quad (z \in \C, \ \bh \in \ell),
\end{equation}
whenever the infinite sum exists. Upholding the above mentioned assumptions about the system and assuming $\bu, \by \in \ell_2$, we can apply Eq.~\eqref{eq:Ztrans} to both sides of Eq.~\eqref{eq:timeDom} to obtain the frequency domain representation of the system, i.e.,
$$
    Y(z) = H(z)U(z) \quad (z \in \overline{\D}).
$$
Under the assumption in Eq.~\eqref{eq:impresp}, the transfer function $H(z) := \cZ[\bh](z)$ belongs to a finite-dimensional subspace of $H_2(\D)$ (see subsection~\ref{subsec:kernAprEx}). The transfer function describes the behavior of the system completely and can be written as
\begin{equation}
    \label{eq:transferFunc}
    H(z) = \sum_{j=0}^{D-1} \frac{w_j}{1 - \overline{\lambda_j}z}.
\end{equation}
By Eq.~\eqref{eq:transferFunc}, the transfer function of a SISO LTI system is a kernel model, where the basis functions consist of Cauchy kernel slices (see Eq.~\eqref{eq:CauchyKer}). The model is described the parameters $w_j$ and $\lambda_j$ $(j=0,\ldots,D-1)$. We note that the transfer function in Eq.~\eqref{eq:transferFunc} differs from the definition commonly found in systems theory literature~\cite{szidar, vandenhof}. Indeed, the poles of $H$ are given by $1/\overline{\lambda_j} \ (j=0,\ldots,D-1)$ and therefore the parameters $\lambda_j$ can be interpreted as the mirror image reflections of the poles across $\T$. If $H$ from Eq.~\eqref{eq:transferFunc} is real rational, then $\hH(z) := H(1/ {z}) \ (z \in \C \setminus \overline{\D})$ is what is usually found in the literature as the transfer function of such a system. The difference between this and our notion of the transfer function in Eq.~\eqref{eq:transferFunc} is due to the considered $\cZ$-transform (see Eq.~\eqref{eq:Ztrans}), which is defined with non-negative exponents (instead of non-positive ones as usual).

We are now ready to pose the problem of finding the parameters $w_j$ and $\lambda_j \ (j=0,\ldots,D-1)$ as an adaptive kernel approximation problem. Let $\cX := \D$, $\cY := \C$ and consider the dataset $\{ (z_k, y_k) \}_{k=1}^{q} \subset \T \times \C \ (q \in \N)$, where $H_{|\T}(z_k) = y_k$. Notably, this means that our dataset is not sampled over $\D$. Fortunately, the transfer function $H$ from Eq.~\eqref{eq:transferFunc} can be analytically continued to $\T$. Thus, identifying the parameters using data sampled on $\T$ is a feasible problem. This is relevant because in practical applications only $H_{|\T}$ can usually be estimated (or measured)~\cite{stenman2000adaptive}. The restriction $H_{|\T}$ is referred to as the frequency response of the system. Define the error functional $E$ as
\begin{multline}
    \label{eq:lsqErr}
    E\left(\left(z_1, f^{\Lambda}(z_1), y_1\right), \ldots, \left(z_q, f^{\Lambda}(z_q), y_q\right)\right) \\ =
    \sum_{k=1}^{q} \left| f^{\Lambda}(z_k) - y_k  \right|^2,
\end{multline}
where the adaptive kernel model $f^{\Lambda}$ is defined according to Eq.~\eqref{eq:adaptiveKernelModels}. We consider two different adaptive kernel models:
\begin{enumerate}[label = \alph*)]
    \item In the first case, we minimize the error functional in Eq.~\eqref{eq:lsqErr} using a Takenaka-Malmquist kernel model, i.e.,
    \begin{equation}
        \label{eq:MTmodel}
        f^{\ba} := \sum_{j=0}^{D-1} b_j \varphi_j^{\ba}, 
    \end{equation}
    where the basis functions $\varphi_j^{\ba} \ (j=0,\ldots,D-1)$ are defined according to Eq.~\eqref{eq:mtsys}. In other words, we consider $\cL := \D^D$ and the parametrization $\Lambda \mapsto \myspan\{ \varphi_k^{\Lambda} : j=0,\ldots,D-1 \} =: \cH(\Lambda) \ (\Lambda \in \cL)$. That is, the RKHSs $\cH(\Lambda)$ coincide with the $D$-dimensional subspaces of $H_2(\D)$. Note that if  
    \begin{equation}
        \label{eq:mtpars}
        \Lambda = \ba := (\lambda_1, \ldots, \lambda_{D-1}) \in \D^D,
    \end{equation} then
    $$
        H \in \myspan\{ \varphi_0^{\ba}, \ldots, \varphi_{D-1}^{\ba} \}
    $$
    and the transfer function from Eq.~\eqref{eq:transferFunc} can be reconstructed perfectly. In addition, choosing this adaptive kernel model can be used to identify the parameters $\lambda_j$ $(j=0,\ldots,D-1)$ in Eq.~\eqref{eq:transferFunc}. Indeed, perfect reconstruction is possible if Eq.~\eqref{eq:mtpars} holds. In this way, optimized model parameters can be used to identify the mirror image poles (and thus the poles) of the transfer function. In this scenario, the model parameters $b_j$ and $a_j$ from Eq.~\eqref{eq:MTmodel} are randomly initialized, and are determined during training with ADAM~\cite{kingma2014adam} by minimizing Eq.~\eqref{eq:lsqErr}. We note that the Wirtinger derivatives of the error function $E$ with respect to the model parameters exist and can be easily computed.

    \item In the second case, we use an RFF model as defined in Eq.~\eqref{eq:RFFmodel}. That is, we consider an adaptive kernel model defined by the parameter space $\cL := \R^Q$ and the parametrization 
    $\Lambda := \left[\omega_0,\ldots,\omega_{Q-1}\right]^\top \mapsto \myspan\{\varphi^\Lambda_j : j=0,\ldots,Q-1 \} =: \cH(\Lambda) \ (\Lambda \in \cL)$, where 
    \begin{equation*}\varphi^\Lambda_j:[-\pi,\pi) \to \C , \quad \varphi^\Lambda_j(t) = \eu^{\iu\omega_j t}  \quad (j=0,\ldots,Q-1).
    \end{equation*}
    In this case, the subspace $\cH(\Lambda)$ will not fully contain the transfer function $H$ for any choice of $Q \in \N$. We demonstrate that even for large values of $Q$, the reconstruction error (in terms of Eq.~\eqref{eq:lsqErr}) may remain large. In addition, the learned parameters $\omega_j \ (j=0,\ldots,Q-1)$  will not contain any information about the system parameters $\lambda_j \ (j=0,\ldots,D-1)$. The parameters $\omega_j$ are initialized using the stochastic approach proposed in the RFF method (see subsection~\ref{subsec:prevMeth}), and then refined using ADAM. In this way, this second approach also uses an adaptive kernel model, albeit with a sub-optimal choice of basis functions. 
\end{enumerate}

For this experiment, we consider an LTI system described by $D=4$ parameters. The true nonlinear parameters are given by
\begin{equation}
    \label{eq:trueMTpars}
    \Lambda^{\ast} := \{ 0.8, \ 0.4 + 0.3\iu, \ 0.4 - 0.3 \iu, -0.5 \} \subset \mathbb{D}
\end{equation}
and the transfer function is defined as
\begin{multline}
    \label{eq:trfex}
        H(z) := \frac{1}{1 - 0.8  z} + \frac{1+\iu}{1 - (0.4 + 0.3 \iu) z} \\ + \frac{1-\iu}{1 - (0.4 - 0.3 \iu)  z} + \frac{1}{1 + 0.5 z}.
\end{multline}

In our experiment, the TM function-based adaptive kernel model from Eq.~\eqref{eq:MTmodel} uses $D=4$ basis functions. This way, if the  parameters from Eq.~\eqref{eq:trueMTpars} are learned correctly, that is, Eq.~\eqref{eq:mtpars} holds, where $\lambda_j \ (j=0,\ldots,3)$ match exactly the values in Eq.~\eqref{eq:trueMTpars}, then the transfer function can be reconstructed perfectly. This is only possible because the parameterized model searches through subspaces spanned by TM functions, one of which contains the exact solution to the learning problem. On the other hand, using the trigonometric model from~\eqref{eq:RFFmodel}, we need much larger subspaces to get comparable approximations to $H$. In our experiments, the trigonometric models use $Q\in\{200,400\}$ parameters. In Fig.~\ref{fig:LTI}, the reconstructed frequency responses are illustrated. For this experiment, we consider a dataset consisting of $q=5000$ samples. The results clearly show that not only the proposed TM function-based adaptive kernel model is able to provide a better approximation using much fewer parameters than the RFF-like approach, but the learned parameters also coincide with the system's true parameters $\lambda_j$. This result is illustrated in Fig.~\ref{fig:poles}. This experiment verifies that choosing an adaptive kernel model based on a priori information about the problem can yield much smaller, better quality, and more meaningful models.

\begin{figure}
    \centering
    \includegraphics[width=0.98\linewidth]{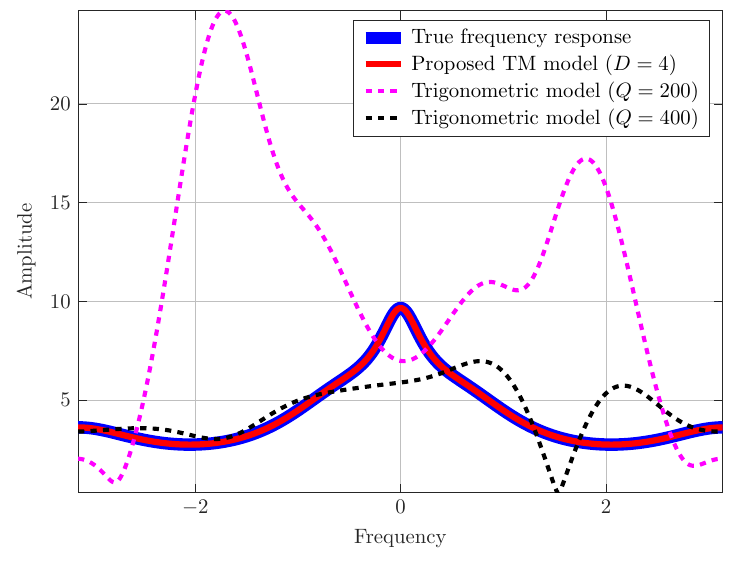}
    \includegraphics[width=0.98\linewidth]{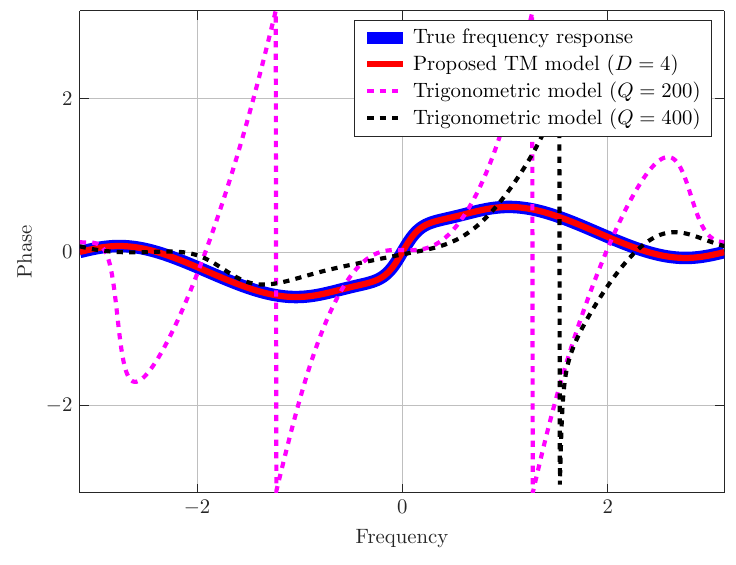}
    \caption{True and reconstructed frequency response $H_{|\T}$ using trigonometric and TM based adaptive kernel models. TOP: absolute value of frequency response. BOTTOM: phase of frequency response.}
    \label{fig:LTI}
\end{figure}

\begin{figure}
    \centering
    \includegraphics[width=0.98\linewidth]{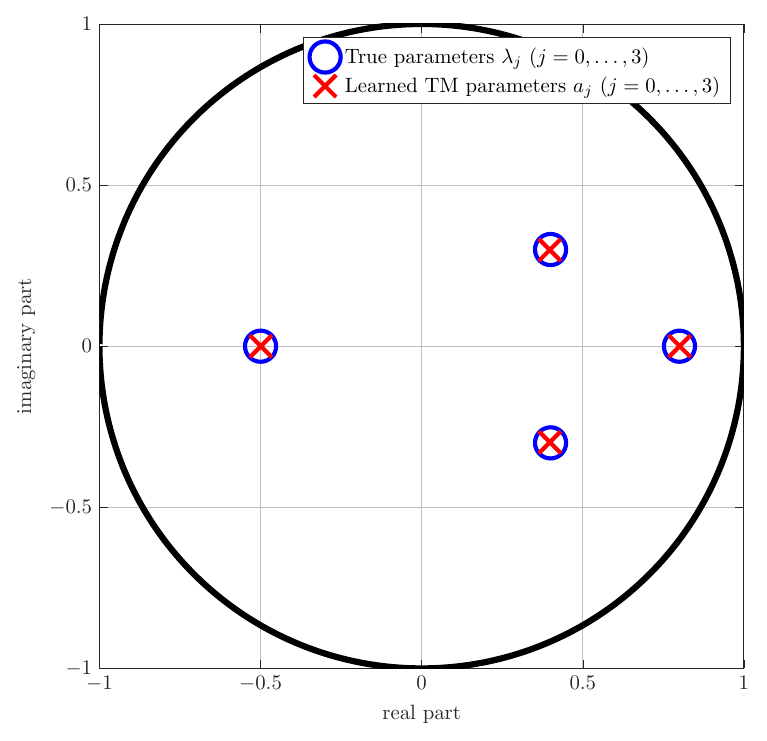}
    \caption{True nonlinear parameters (mirror image poles) of the transfer function $H$ and the learned parameters of the TM based adaptive kernel model.}
    \label{fig:poles}
\end{figure}

\subsection{Forest cover type classification}
\label{subsec:benchmark}

In our second set of experiments, we demonstrate that the proposed adaptive kernel models can provide state-of-the-art performance on benchmark problems. We select the  ForestCover dataset~\cite{CVM} to demonstrate this. This dataset contains measurements about trees from four areas of the Roosevelt National Forest in Colorado. The data consists of $54$-dimensional vectors, representing measurements about forest covers at certain locations. It contains a total of $581012$ data samples. In this case, the map-to-be-learned $F : \R^{54} \supset \cX \to \cY := \{0,1,\ldots,6 \}$ can be interpreted as a classifier that maps the data samples to different forest cover types. This problem is highly nonlinear, and no a priori assumptions can be made on the structure of $F$ to help us select an appropriate family of parameterized maps (like in the system identification example of subsection~\ref{subsec:lti}). However, previous studies~\cite{rahimi, CVM} indicate that  radial basis function-based kernels and their approximations (see, e.g.,~\cite{rahimi}) can been used to solve this classification problem with a high precision. For this reason, we first consider a trigonometric adaptive model. That is, we consider a model $f^{\Lambda} : \R^{54} \to \R^{7}$, where
\begin{equation}
\label{eq:classMod}
f^{\Lambda}_{\mu} = \sum_{j=0}^{D-1} w_{j,\mu} \eu^{\iu \langle \lambda_j, \cdot \rangle_{\R^N}} \quad (\mu=0,\ldots,6)
\end{equation}
for different choices of $D \in \N$. The frequency parameters $\lambda_j \in \R^{54} \ (j=0,\ldots,D-1)$ are optimized in two different ways:
\begin{enumerate}[label = \alph*)]
    \item Using the RFF method (see subsection~\ref{subsec:prevMeth}): The parameters $\lambda_j \ (j=0,\ldots,D-1)$ are picked randomly according to an appropriate probability distribution and are then fixed. Training in this case consists of finding optimal parameters $W=[w_{j,\mu}]_{j=0,\mu=0}^{D-1,6}\in \R^{D\times 7}$. 
    \item Proposed method: The parameters $\lambda_j \ (j=0,\ldots,D-1)$ are initialized in a uniformly random manner over the interval $[-10, 10]^{54} \subset \R^{54}$ and are then improved iteratively using ADAM. In this case both $\lambda_j$ and $w_j\in\R^7$ are considered free parameters.
\end{enumerate}
In both cases, the  parameters $w_{j,\mu} \ (j=0,\ldots,D-1, \ \mu=0,\ldots,6)$ are optimized using ADAM  with a batch size of $8000$. Thus, both the considered RFF and proposed trigonometric models can be interpreted as adaptive kernel methods relying on the parameter space $\cL := \R^{54 \times D}$ and the parametrization $\Lambda \mapsto \myspan\{ \eu^{\iu \langle \lambda_j, \cdot \rangle_{\R^N}} : j=0,\ldots,D-1 \} =: \cH(\Lambda) \ (\Lambda = [\lambda_0, \ldots, \lambda_{D-1}]^\top \in \cL)$. The difference between the two models is that for the RFF case, we pick a $\Lambda \in \cL$ according to the stochastic method proposed in~\cite{rahimi} and do not modify this choice throughout the training. In contrast, when considering the proposed trigonometric adaptive model, the parameters $\Lambda \in \cL$ defining the space $\cH(\Lambda)$ are trained together with the parameters  $w_{j,\mu} \ (j=0,\ldots,D-1, \ \mu=0,\ldots,6)$ using gradient-based optimization techniques.

In addition to the trigonometric model in Eq.~\eqref{eq:classMod}, we also consider another adaptive kernel model using non-trigonometric basis functions. 
The considered kernel model can be written as
\begin{equation}
\label{eq:classModRat}
g^{\Lambda}_{\mu}(x) = \sum_{j=0}^{D-1} w_{j,\mu} \arctan ( \langle x, \lambda_j \rangle_{\R^{N}}),
\end{equation}
where $\lambda_j \in \R^N$ with $N = 54$ and $w_{j,\mu} \in \R$ $(j=0,\ldots,D-1, \ \mu=0,\ldots,6)$. Thus an adaptive kernel model is obtained defined by the parameter space $\cL := \R^{N \times D}$ and the parametrization $\Lambda \mapsto \myspan\{  \arctan ( \langle x, \lambda_j \rangle_{\R^{N}} : j=0,\ldots,D-1 \} =: \cH(\Lambda) \ (\Lambda \in \cL)$. Similarly to the trigonometric model, the parameters of this arc tangent (AT) basis adaptive kernel model are optimized using ADAM as specified below.
 
For this experiment, the error functional is chosen as the Crammer-Singer multi-class hinge loss function~\cite{crammer2001algorithmic} defined as
\begin{multline}
    \label{eq:mhinge}
    E\left(\left(t_1, f^{\Lambda}(t_1), y_1\right), \ldots, \left(t_q, f^{\Lambda}(t_q), y_q\right)\right) \\ = 
    \frac{1}{q} \sum_{k=1}^{q} \max_{\mu \in \left\{0,\ldots,6\right\} \setminus \{y_n\}}  \max\left\{0,  1 +  f^{\Lambda}_\mu(t_k) - f^{\Lambda}_{y_k}(t_k)\right\} \\ + 
    \alpha  {\| C \|}_{\rm F} 
\end{multline}
for $\alpha = 0.1$, where ${\|\cdot\|}_{\rm F}$ denotes the Frobenius norm.
Each considered model (RFF, adaptive trigonometric, adaptive AT) is trained using a batch size of $8000$ with an epoch budget of $500$. The ForestCover dataset is split into disjoint training and test sets comprising $80 \%$ and $20 \%$ percent of the dataset, respectively. In this way, the training set contains a total of $q = 464809$ data samples, while the generalization ability of the model is measured using $116203$ samples in the test set. Model efficiency is measured in terms of classification accuracy. The performance of the proposed (adaptive trigonometric and adaptive AT) models and the implemented RFF scheme are illustrated in Table~\ref{tab:forestCov}. The last two columns denote the training time and number of epochs required to reach the maximum accuracy on the test set within the epoch budget. It also contains reported achieved accuracies on the same dataset by other models in the literature. 

\begin{table}[bt]
    \centering
    \caption{Classification performance of the proposed adaptive kernel models (trigonometric and arc tangent), RFF, and approaches from the literature.}
    \begin{tabular}{ccccc} \toprule
        \textbf{method} & \textbf{$D$} & \makecell{\textbf{accuracy} \\ \textbf{($\%$)}} & \makecell{\textbf{train. time} \\ \textbf{(min)}} & \makecell{\textbf{reference}\\ \textbf{epoch}} \\ \midrule
        \multirow{4}{*}{ \makecell{adapt. trig. \\ (proposed)}} &  $500$ & $81.6$ & $70.0$ & $495$ \\ 
        & $1000$ & $86.9$ & $86.48$ & $499$\\ 
        & $5000$ & $92.2$ & $32.27$ & $181$ \\ 
        & $10000$ & $92.7$ & $30.00$ & $157$ \\ \midrule
        \multirow{4}{*}{\makecell{adapt. AT \\ (proposed)}} & $500$ & $79.4$ & $85.9$ & $495$\\ 
        & $1000$ & $82.0$ & $83.8$ & $494$ \\ 
        & $5000$ & $86.5$ & $85.1$ & $493$\\ 
        & $10000$ & $87.3$ & $85.0$ & $479$ \\ \midrule
        \multirow{4}{*}{\makecell{RFF \\ (implemented)}} & $500$ & $77.2$ & $82.3$ & $484$\\ 
        & $1000$ & $78.8$ & $82.5$ & $484$\\ 
        & $5000$ & $81.0$ & $82.8$ & $455$\\ 
        & $10000$ & $81.3$ & $94.4$ & $493$ \\ \midrule 
        RFF~\cite{rahimi} & $5000$ & $88.4$ & $71$ & -\\ 
        exact SVM~\cite{rahimi} & $464809$ & $97.8$ & $2640$ & - \\ \bottomrule
    \end{tabular}
    \label{tab:forestCov}
\end{table}

Our results clearly illustrate the benefit of using adaptive kernel models. The proposed adaptive models (especially the one using trigonometric basis functions) clearly outperform our implementation of the RFF method. In fact, using only $D=500$ basis functions, it  results in $\approx 4 \%$ improved accuracy, when compared to an RFF model of same dimension. The training times show that despite the significant improvement in model performance and gradient-based minimization of Eq.~\eqref{eq:mhinge}, the training costs do not explode compared to the RFF method. In fact, in the case of the adaptive trigonometric model, much fewer epochs are needed to find a good approximation of $F$ and the training time is significantly reduced compared to the RFF method. 
The AT model also outperforms the RFF approach, but the results show that models satisfying Eq.~\eqref{eq:classMod} are a more appropriate choice for this particular example.

The last two lines of Table~\ref{tab:forestCov} illustrate the performance of similar kernel models as reported in the literature. We note that these methods are not directly implemented by us. In the case of the RFF method, the reported accuracy is significantly higher than the one achieved by our own implementation. Presumably, this is due to the differences in the training setup (different training and test sets, different optimization scheme, etc.). Despite these differences, the proposed adaptive trigonometric model is able to outperform the RFF method in terms of accuracy for the same amount of parameters ($D=5000$).

For this experiment, the proposed adaptive trigonometric model (Eq.~\eqref{eq:classMod}) provides the best performance, closely followed by the non-trigonometric AT model (Eq.~\eqref{eq:classModRat}). These significantly outperform the RFF (and the remaining) methods. It can be seen that by choosing $D=1000$, the proposed models achieve $>80 \%$ classification accuracy. Such a performance is only achieved for $D=5000$ when using our own RFF implementation as well as the original implementation described in~\cite{rahimi}. Increasing the dimension of the solution space significantly improves the adaptive models, with the trigonometric model achieving over $90 \%$ accuracy on the test set with $D=5000$. This experiment shows that learning the correct solution space can significantly reduce the number of required model parameters, and therefore the inference time of the classifier, while improving model performance. 
Additionally, when the parametrization $\Lambda\mapsto\cH (\Lambda)$ is chosen well, then the use of adaptive models can significantly reduce training time. 
This, together with the findings in subsection~\ref{subsec:lti}, demonstrates that the proposed adaptive kernel models can play a crucial role in the development of interpretable, low-complexity, and high-performing kernel methods.

\section{Conclusions}
\label{sec:conc}

In this paper, we considered kernel models using parameterized bases. We demonstrated that choosing the right bases can improve expansion-based approximations to kernels of infinite-dimensional RKHSs. We also verified that adaptive kernel models can significantly improve the performance of previous kernel methods despite containing fewer (and sometimes physically meaningful) parameters. We specified sufficient conditions for the existence of optimal adaptive kernel models and studied their properties.

In the next phase of our research, we will develop applications using the modeling procedure proposed in this paper. One possible application could be to use the proposed adaptive kernel models to approximate the nonlinear dynamics of control systems. The produced approximation could then be used by the controller to synthesize an optimal input for the system.

\section*{Code and data availability}
The Python implementation of the proposed methods and experiments can be downloaded from 
\begin{center}
\url{https://gitlab.com/tamasdzs/adaptive-kernel-methods}.
\end{center}

\section*{Acknowledgment}
The research of TD has received funding from the Swiss
Government Excellence Scholarship No. 2025.0057.  The
research of AA was funded in part by the Swiss National
Science Foundation (SNSF) grant No. 224943. The research of ZS and JB was supported by the European Union within the framework of the National Laboratory for Autonomous Systems  (RRF-2.3.1-21-2022-00002).

\section*{CRediT author statement}
\textbf{Tam\'as D\'ozsa:} Conceptualization, Methodology, Software, Formal analysis, Investigation, Writing -- original draft;
\textbf{Andrea Angino:} Methodology, Software, Writing -- Original Draft, Visualization; 
\textbf{Zolt\'an Szab\'o:} Validation, Writing -- Review \& Editing;
\textbf{J\'ozsef Bokor:} Conceptualization, Funding acquisition; 
\textbf{Matthias Voigt:} Writing -- Review \& Editing, Supervision, Project administration, Funding acquisition.
\ifCLASSOPTIONcaptionsoff
  \newpage
\fi

\bibliographystyle{ieeetr}
\bibliography{refs}

\end{document}